\documentclass[11pt]{amsart}
\usepackage{amsmath,amssymb,amsthm}
\usepackage[margin=1in]{geometry}

\newtheorem{theorem}{Theorem}[section]
\newtheorem{lemma}[theorem]{Lemma}
\newtheorem{proposition}[theorem]{Proposition}
\newtheorem{definition}[theorem]{Definition}
\newtheorem{remark}[theorem]{Remark}
\newtheorem{example}[theorem]{Example}

\begin{document}
\title[Weyl Pseudo Almost Periodic Type Solutions to SSEEs driven by fBm]{Weyl Pseudo Almost Periodic Type Solutions to Semilinear Stochastic Evolution Equations Driven by Fractional Brownian Motion}

\author{Dimplekumar N. Chalishajar}
\address{Carroll Hall, Virginia Military Institute (VMI), VA-24450, USA}
\email{chalishajardn@vmi.edu}

\author{Marko Kosti\' c}
\address{Faculty of Technical Sciences,
University of Novi Sad,
Trg D. Obradovi\' ca 6, 21125 Novi Sad, Serbia}
\email{marco.s@verat.net}

\author{Daniel Velinov}
\address{Department for Mathematics and Informatics, Faculty of Civil Engineering, Ss. Cyril and Methodius University in Skopje,
Partizanski Odredi
24, P.O. box 560, 1000 Skopje, N. Macedonia}
\email{velinovd@gf.ukim.edu.mk}

{\renewcommand{\thefootnote}{} \footnote{
\\ \text{  }  \ \    {\it Key words and phrases.} Weyl almost periodic solutions, Weyl double-measure pseudo almost periodic solutions, fractional Brownian motion, stochastic evolution equations.
\\  \text{  }  \ \ This research is partially supported by grant 174024 of Ministry of Science and Technological Development, Republic of Serbia and Bilateral project between MANU and SANU.}}

\maketitle

\begin{abstract} 
In this paper, we analyze square-mean Weyl almost periodic solutions and square-mean Weyl double-measure pseudo almost periodic solutions for a class of semilinear evolution equations in separable Hilbert spaces driven by two-sided fractional Brownian motion with Hurst index $H<1/2$. Due to the non-integrability of covariance density for $H<1/2$, a H\"{o}lder-continuity condition on the diffusion coefficient is required. An illustrative example involving a stochastic parabolic equation demonstrates the applicability of obtained results.
\end{abstract}

\section{Introduction}

Bohr almost periodic functions represent a natural generalization of periodic functions, and the study of almost periodic solutions of differential equations, in a variety of senses, has been one of the central themes of the qualitative theory of differential equations since the appearance of Bohr's paper \cite{bohr}. The class of Weyl almost periodic functions, introduced  by Weyl \cite{Weyl}, generalizes both the Bohr and the Stepanov classes; see Andres, Bersani, Grande \cite{AndresBersaniGrande}, Kosti\' c \cite{k0} and Levitan \cite{Levitan} for more details. The existence results for Weyl almost periodic solutions of linear and semilinear abstract equations have been obtained under suitable restrictions by Bedouhene, Ibaouene, Mellah, and de Fitte \cite{BedouheneEtAl}, and the Weyl framework has since been extended in several directions, including the general-metric setting of Kosti\'c \cite{k1} and the $c$-almost periodic classes of Ounis and Sepulcre \cite{OunisSepulcre}; connections between Weyl almost periodicity and the spectral theory of dynamical systems have further been explored by Lenz, Spindeler, and Strungaru \cite{LenzSpindelerStrungaru}.

A second, largely independent line of development concerns pseudo almost periodicity, introduced by Zhang \cite{zhang} as a decomposition of a bounded continuous function into an almost periodic part and an ergodic perturbation. Among the many refinements of this notion, the double-measure pseudo almost periodicity of Blot, Cieutat, and Ezzinbi \cite{BlotCieutatEzzinbi,BlotCieutatEzzinbiApplic}, itself building on the weighted pseudo almost automorphic framework of Blot, N'Gu\'er\'ekata, and Pennequin \cite{BlotNguerekataPennequin}, and further developed by Diagana, Ezzinbi, and Miraoui \cite{DiaganaEzzinbiMiraoui}, is the most comprehensive to date, replacing the classical Lebesgue-measure ergodicity condition by a pair of comparable positive measures $\mu,\nu$ subject to a translation-quasi-invariance condition. Until recently, however, no theory of Weyl double-measure pseudo almost periodic functions existed, the incompleteness of the underlying Weyl space having obstructed the construction of the corresponding fixed-point machinery. This gap was closed, at the deterministic level, by introducing the Weyl double-measure pseudo almost periodic functions via the Weyl seminorm and establishing existence and global exponential stability of such solutions for a class of semilinear evolution equations governed by an exponentially stable semigroup \cite{LiWeyl}; that work is the immediate deterministic antecedent of the present paper, and we adopt its function-space framework throughout, filling in, along the way, several details left implicit there, in the square-mean setting developed below.

Independently, the study of almost periodicity for stochastic processes and stochastic differential equations has developed around the square-mean sense introduced by Bezandry and Diagana \cite{BezandryDiagana}, in which almost periodicity is required of the map $t\mapsto X(t)$ regarded as taking values in the space of square-integrable random variables, rather than along individual sample paths, the latter notion being in general too strong to hold for solutions of genuinely noise-driven equations; the classical theory of composition and superposition operators for almost periodic functions of two variables, on which such results ultimately rest, goes back to Fink \cite{Fink}. This square-mean framework has been developed extensively for the Bohr and Stepanov classes, and mild solutions of stochastic evolution equations in Hilbert space are, throughout, understood in the sense of the variation-of-constants formula associated with the semigroup, as in Da Prato and Zabczyk \cite{DaPratoZabczyk}. To the best of our knowledge, the square-mean framework has not previously been combined with the Weyl seminorm, nor with the double-measure pseudo almost periodicity of \cite{BlotCieutatEzzinbi}, leaving open the stochastic counterpart of the theory built in \cite{LiWeyl}.

A further motivation for the present paper comes from the study of stochastic equations driven by fractional Brownian motion, a centered Gaussian process exhibiting long-range dependence and self-similarity, of considerable interest in applications ranging from mathematical finance to hydrology and turbulence, and studied systematically from the stochastic-calculus point of view by Duncan, Hu, and Pasik-Duncan \cite{DHP}, by Nualart \cite{nualart}, and, in infinite dimensions, by Boufoussi and Hajji \cite{BoufoussiHajji} and Mishura \cite{Mishura}. Exponential-stability results for stochastic delay evolution equations driven by fractional Brownian motion, in the spirit of the stability theorem obtained below, have been established by Caraballo, Garrido-Atienza, and Taniguchi \cite{CaraballoGarridoTaniguchi}. When the Hurst index satisfies $H<1/2$, the covariance density of fractional Brownian motion fails to be locally integrable at the origin, and the resulting Wiener integral must be estimated through a regularized covariance formula that requires the integrand to satisfy a H\"older-continuity condition of order strictly greater than $1/2-H$; this requirement has no counterpart in the classical It\^o theory and, as we show below, interacts nontrivially with the exponential decay of the underlying semigroup. The closest existing result combining almost periodicity with fractional Brownian motion is due to Li and Li \cite{LiLi}, who established the existence of Weyl almost periodic solutions in distribution to a mean-field stochastic differential equation driven by fractional Brownian motion; their framework, however, is confined to the Weyl almost periodic class in the sense of distribution, does not treat the double-measure pseudo almost periodic case, and does not address stability.

The present paper combines these three lines of development. We work in a separable Hilbert space $X$ and consider the semilinear stochastic evolution equation
\begin{align}\label{maineq}
dX(t)=AX(t)\,dt+f(t,X(t))\,dt+\sigma(t,X(t))\,dB^H(t),\qquad t\in\mathbb R,
\end{align}
where $A$ generates an exponentially stable $C_0$-semigroup and $B^H$ is a two-sided fractional Brownian motion of Hurst index $H\in(0,1/2)$. The novelty of this work is threefold: first, we introduce the notions of square-mean Weyl almost periodicity and square-mean Weyl double-measure pseudo almost periodicity for Hilbert-space-valued stochastic processes, extending simultaneously the Weyl-seminorm framework of \cite{LiWeyl} and the square-mean framework of \cite{BezandryDiagana}, and, in doing so, we supply complete proofs, in the square-mean setting, of the function-space properties that the deterministic theory leaves as routine but unrecorded; second, we establish, by means of a regularized covariance estimate for the stochastic convolution against fractional Brownian motion with $H<1/2$, the existence and uniqueness of mild solutions of each type, together with their global mean-square exponential stability, for the equation above, with every step of the argument, including the composition and fixed-point estimates, carried out in full detail; and third, we make explicit the resulting contraction constant and show precisely how it depends on the Hurst index $H$, quantifying the price, in terms of the admissible Lipschitz constant of the diffusion coefficient, of passing from standard Brownian motion to fractional Brownian motion with small Hurst index. No prior results of this type exist in the literature.

The paper is organized as follows. Section 2 collects the deterministic and stochastic preliminaries required for the paper, including the classical Weyl almost periodic function spaces, the double-measure ergodic class of \cite{BlotCieutatEzzinbi}, their square-mean analogues, and the covariance estimate for the stochastic convolution against fractional Brownian motion with $H<1/2$; every structural property of these spaces is stated as a theorem. Section 3 contains the main results of the paper: the existence and uniqueness of square-mean Weyl almost periodic and square-mean Weyl double-measure pseudo almost periodic mild solutions, and their global mean-square exponential stability. Section 4 presents an illustrative example, drawn from a stochastic parabolic equation on a bounded interval, showing that the hypotheses of the main theorems are easily verified. Section 5 concludes the paper.

\section{Preliminaries}

Throughout the paper, $(X,\|\cdot\|)$ denotes a separable real Hilbert space and $(\Omega,\mathcal F,p,\{\mathcal F_t\}_{t\in\mathbb R})$ a complete filtered probability space, where, following the usual convention, $p$ denotes the probability measure. We denote by $H:=L^2(\Omega,X)$ the space of $X$-valued random variables $x$ with $E\|x\|^2<\infty$, equipped with the inner product $\langle x,y\rangle_H:=E\langle x,y\rangle_X$ and norm $\|x\|_H:=(E\|x\|^2)^{1/2}$; $(H,\langle\cdot,\cdot\rangle_H)$ is itself a separable Hilbert space. We denote by $L^2_{loc}(\mathbb R,X)$, respectively $L^2_{loc}(\mathbb R,H)$, the corresponding spaces of locally square-integrable functions.

\subsection{The deterministic Weyl seminorm}

For $\varphi\in L^p_{loc}(\mathbb R,X)$, $1\le p<\infty$, the Weyl seminorm of $\varphi$ is defined by
\begin{align*}
\|\varphi\|_{W^p}=\lim_{\flat\to\infty}\sup_{\vartheta\in\mathbb R}\left(\frac1\flat\int_{\vartheta}^{\vartheta+\flat}\|\varphi(s)\|^p\,ds\right)^{1/p}.
\end{align*}
A function $\varphi$ is $W^p$-continuous if $\lim_{h\to0}\|\varphi(\cdot+h)-\varphi(\cdot)\|_{W^p}=0$, and $W^p$-bounded if $\|\varphi\|_{W^p}<\infty$; the corresponding spaces are denoted $W^p_{B}(\mathbb R,X)$ and $W^p_{BC}(\mathbb R,X)$. A function $\varphi\in W^p_{B}(\mathbb R,X)$ is $p$-th Weyl almost periodic if, for every $\epsilon>0$, the set $E\{\epsilon,\varphi\}=\{\varsigma\in\mathbb R:\|\varphi(\cdot+\varsigma)-\varphi(\cdot)\|_{W^p}<\epsilon\}$ is relatively dense in $\mathbb R$; the family of such functions is denoted by $W^p_{ap}(\mathbb R,X)$, and $W^p_{AP}(\mathbb R,X)=W^p_{ap}(\mathbb R,X)\cap W^p_{BC}(\mathbb R,X)$. The space $(W^p_{AP}(\mathbb R,X),\|\cdot\|_{W^p})$ is a complete linear space; this fact, together with translation invariance and the preservation of Weyl almost periodicity under composition with Lipschitz maps, is established in \cite[Section 2]{LiWeyl} for a general Banach space $X$, and we recall these results for reference. In Section 2.2 below we establish the corresponding statements for the square-mean spaces used throughout the remainder of the paper.

\subsection{Double-measure pseudo almost periodicity}

Let $\mathcal F$ denote the Lebesgue $\sigma$-field of $\mathbb R$ and let $\mathcal M^+$ be the set of all positive measures $\mu$ on $\mathcal F$ satisfying $\mu(\mathbb R)=+\infty$ and $\mu([a,b])<+\infty$ for all $a\le b$. Two measures $\mu,\nu\in\mathcal M^+$ are equivalent, written $\mu\sim\nu$, if there exist positive constants $\alpha_0,\beta_0$ and a bounded interval $A$ such that $\alpha_0\nu(B)\le\mu(B)\le\beta_0\nu(B)$ for all $B\in\mathcal F$ with $B\cap A=\emptyset$. For $\mu\in\mathcal M^+$ and $\sigma_0\in\mathbb R$, the shifted measure $\mu_{\sigma_0}$ is defined by $\mu_{\sigma_0}(A)=\mu(\{a+\sigma_0:a\in A\})$. We impose the following two hypotheses, standard in this theory \cite{BlotCieutatEzzinbi,DiaganaEzzinbiMiraoui}:

$(F1)$ $\mu,\nu\in\mathcal M^+$ satisfy $\limsup_{h\to\infty}\mu([-h,h])/\nu([-h,h])<\infty$.

$(F2)$ For every $\mu\in\mathcal M^+$ and $\sigma_0\in\mathbb R$, there exist a constant $\eta_0>0$ and a bounded interval $\Gamma\subset\mathbb R$ such that $\mu(a+\sigma_0:a\in A)\le\eta_0\mu(A)$ whenever $A\in\mathcal F$ and $A\cap\Gamma=\emptyset$.

Under $(F1)$ and $(F2)$, $\mu\sim\mu_{\sigma_0}$ for all $\sigma_0\in\mathbb R$; this fact, purely measure-theoretic and independent of the space $X$, is established in \cite{BlotCieutatEzzinbi} and is used below without modification. For $\mu,\nu\in\mathcal M^+$, a function $\varphi\in W^p_{BC}(\mathbb R,X)$ belongs to the ergodic class $E^p(\mathbb R,X,\mu,\nu)$ if
\begin{align*}
\lim_{h\to+\infty}\frac{1}{\nu([-h,h])}\int_{-h}^{h}\left(\lim_{\flat\to+\infty}\frac1\flat\int_{\vartheta}^{\vartheta+\flat}\|\varphi(s)\|^p\,ds\right)^{1/p}d\mu(\vartheta)=0,
\end{align*}
and a function $\varphi\in W^p_{BC}(\mathbb R,X)$ is $p$-th Weyl $(\mu,\nu)$-pseudo almost periodic if it can be written as $\varphi=\varphi_1+\varphi_2$ with $\varphi_1\in W^p_{AP}(\mathbb R,X)$ and $\varphi_2\in E^p(\mathbb R,X,\mu,\nu)$; the resulting space is denoted $W^p_{PAP}(\mathbb R,X,\mu,\nu)$. This decomposition is in general not unique \cite{LiWeyl}, a fact which is illustrated stochastically in Section 4 below.

\subsection{Square-mean Weyl almost periodicity}

For $\varphi\in L^2_{loc}(\mathbb R,H)$, the square-mean Weyl seminorm of $\varphi$ is defined by
\begin{align*}
\|\varphi\|_{W^2}=\lim_{\flat\to\infty}\sup_{\vartheta\in\mathbb R}\left(\frac1\flat\int_{\vartheta}^{\vartheta+\flat}E\|\varphi(s)\|^2\,ds\right)^{1/2}.
\end{align*}

\begin{definition}
A function $\varphi\in L^2_{loc}(\mathbb R,H)$ is square-mean $W^2$-continuous if $\lim_{h\to0}\|\varphi(\cdot+h)-\varphi(\cdot)\|_{W^2}=0$, and $W^2$-bounded if $\|\varphi\|_{W^2}<\infty$. We denote by $W^2_{B}(\mathbb R,X)$ the set of all $W^2$-bounded functions and by $W^2_{BC}(\mathbb R,X)$ the set of all $W^2$-bounded and $W^2$-continuous functions.
\end{definition}

\begin{definition}
A function $\varphi\in W^2_{B}(\mathbb R,X)$ is square-mean Weyl almost periodic if, for every $\epsilon>0$, the set $E\{\epsilon,\varphi\}=\{\varsigma\in\mathbb R:\|\varphi(\cdot+\varsigma)-\varphi(\cdot)\|_{W^2}<\epsilon\}$ is relatively dense in $\mathbb R$. The family of such functions is denoted $W^2_{ap}(\mathbb R,X)$, and we set $W^2_{AP}(\mathbb R,X)=W^2_{ap}(\mathbb R,X)\cap W^2_{BC}(\mathbb R,X)$.
\end{definition}

\begin{proposition}[Reduction principle]\label{reduction}
A function $\varphi\in L^2_{loc}(\mathbb R,X)$, regarded pointwise as an $X$-valued random variable, belongs to $L^2_{loc}(\mathbb R,H)$ if and only if $E\|\varphi(s)\|^2<\infty$ for almost every $s$ and $s\mapsto\varphi(s)$, as a map into $H$, is locally square-integrable, and in that case
\begin{align*}
\|\varphi\|_{W^2}=\|\varphi\|_{W^2_H},
\end{align*}
where the right-hand side is the deterministic Weyl seminorm of Section 2.1, taken with $p=2$ and with $X$ replaced by $H$. Consequently, $\varphi$ is square-mean $W^2$-continuous, $W^2$-bounded, or square-mean Weyl almost periodic, respectively, if and only if it is $W^2$-continuous, $W^2$-bounded, or $2$nd-Weyl almost periodic, respectively, as an $H$-valued function in the sense of Section 2.1.
\end{proposition}

\begin{proof}
Both sides of $E\|\varphi(s)\|_X^2=\|\varphi(s)\|_H^2$ are, by the definition of the norm on $H=L^2(\Omega,X)$, literally the same quantity; substituting this identity into the definitions of $\|\cdot\|_{W^2}$ (square-mean) and $\|\cdot\|_{W^2_H}$ (deterministic, $p=2$, target space $H$) shows the two seminorms coincide term by term, hence in the limit. The stated equivalences of the continuity, boundedness, and almost periodicity properties are then immediate from Definitions 2.1--2.2 above and their $H$-valued counterparts in Section 2.1.
\end{proof}

Proposition \ref{reduction} identifies the square-mean Weyl function spaces built over $X$ with the deterministic Weyl function spaces of Section 2.1 built over the Hilbert space $H=L^2(\Omega,X)$, with exponent $p=2$. Every structural property of the deterministic spaces established in \cite[Section 2]{LiWeyl} is proved there for a general Banach space, hence applies verbatim to $H$; we record the resulting statements for $W^2_{AP}(\mathbb R,X)$ below, together with self-contained proofs, both for completeness and because the composition-type statement (Theorem \ref{Lipcomp}) requires an argument beyond a formal restatement.

\begin{lemma}\label{BClinear}
If $\varphi_1,\varphi_2\in W^2_{BC}(\mathbb R,X)$ and $\alpha$ is a scalar, then $\varphi_1+\varphi_2,\ \alpha\varphi_1\in W^2_{BC}(\mathbb R,X)$.
\end{lemma}

\begin{proof}
By Minkowski's inequality for the $L^2(\Omega)$-norm, for every $\vartheta,\flat$,
\begin{align*}
\left(\frac1\flat\int_\vartheta^{\vartheta+\flat}E\|\varphi_1(s)+\varphi_2(s)\|^2\,ds\right)^{1/2}\le\left(\frac1\flat\int_\vartheta^{\vartheta+\flat}E\|\varphi_1(s)\|^2\,ds\right)^{1/2}+\left(\frac1\flat\int_\vartheta^{\vartheta+\flat}E\|\varphi_2(s)\|^2\,ds\right)^{1/2},
\end{align*}
and taking $\sup_\vartheta$ and $\flat\to\infty$ on both sides gives $\|\varphi_1+\varphi_2\|_{W^2}\le\|\varphi_1\|_{W^2}+\|\varphi_2\|_{W^2}<\infty$, so $\varphi_1+\varphi_2\in W^2_B(\mathbb R,X)$; likewise $\|\alpha\varphi_1\|_{W^2}=|\alpha|\,\|\varphi_1\|_{W^2}<\infty$. For continuity, the same triangle inequality applied to $(\varphi_1+\varphi_2)(\cdot+h)-(\varphi_1+\varphi_2)(\cdot)=[\varphi_1(\cdot+h)-\varphi_1(\cdot)]+[\varphi_2(\cdot+h)-\varphi_2(\cdot)]$ gives $\|(\varphi_1+\varphi_2)(\cdot+h)-(\varphi_1+\varphi_2)(\cdot)\|_{W^2}\le\|\varphi_1(\cdot+h)-\varphi_1(\cdot)\|_{W^2}+\|\varphi_2(\cdot+h)-\varphi_2(\cdot)\|_{W^2}\to0$ as $h\to0$; similarly for $\alpha\varphi_1$.
\end{proof}

\begin{theorem}\label{BCcomplete}
$(W^2_{BC}(\mathbb R,X),\|\cdot\|_{W^2})$ is a complete space.
\end{theorem}

\begin{proof}
Let $\{\varphi_n\}\subset W^2_{BC}(\mathbb R,X)$ be a Cauchy sequence. For every $\epsilon>0$ there is $N(\epsilon)$ such that for all $n,m\ge N$, $\|\varphi_n-\varphi_m\|_{W^2}<\epsilon$, and hence there is $\flat_0(\epsilon)>0$ such that for $\flat\ge\flat_0$,
\begin{align*}
\sup_{\vartheta\in\mathbb R}\left(\frac1\flat\int_\vartheta^{\vartheta+\flat}E\|\varphi_n(s)-\varphi_m(s)\|^2\,ds\right)^{1/2}<2\epsilon.
\end{align*}
Fixing $\flat\ge\flat_0$, this gives $\int_\vartheta^{\vartheta+\flat}E\|\varphi_n(s)-\varphi_m(s)\|^2\,ds<\flat(2\epsilon)^2$ for all $\vartheta\in\mathbb R$ and $n,m\ge N$, so $\{\varphi_n\}$ is Cauchy in $L^2([\vartheta,\vartheta+\flat],H)$ for every $\vartheta$; since $H$ is a Hilbert space, $L^2([\vartheta,\vartheta+\flat],H)$ is complete, so there is $\varphi\in L^2([\vartheta,\vartheta+\flat],H)$ with $\varphi_n\to\varphi$ there. As the limit is unique almost everywhere on any such interval, patching over $\vartheta\in\mathbb R$ gives $\varphi\in L^2_{loc}(\mathbb R,H)=L^2_{loc}(\mathbb R,X)$ (identifying random variables and elements of $H$ as in Proposition \ref{reduction}) with $\varphi_n(t)\to\varphi(t)$ in $L^2_{loc}$.

For each $\flat\ge\flat_0$ and $n\ge N$, letting $m\to\infty$ in the previous display and using Fatou's lemma,
\begin{align*}
\sup_{\vartheta\in\mathbb R}\left(\frac1\flat\int_\vartheta^{\vartheta+\flat}E\|\varphi_n(s)-\varphi(s)\|^2\,ds\right)^{1/2}\le2\epsilon,
\end{align*}
so, letting $\flat\to\infty$, $\|\varphi_n-\varphi\|_{W^2}\le2\epsilon$ for $n\ge N$. Since $\varphi-\varphi_n,\varphi_n\in W^2_B(\mathbb R,X)$ and $\varphi=(\varphi-\varphi_n)+\varphi_n$, Lemma \ref{BClinear} gives $\varphi\in W^2_B(\mathbb R,X)$. Finally, for any $h\in\mathbb R$,
\begin{align*}
\|\varphi(\cdot+h)-\varphi(\cdot)\|_{W^2}\le\|\varphi(\cdot+h)-\varphi_n(\cdot+h)\|_{W^2}+\|\varphi_n(\cdot+h)-\varphi_n(\cdot)\|_{W^2}+\|\varphi_n(\cdot)-\varphi(\cdot)\|_{W^2},
\end{align*}
and letting $n\to\infty$ then $h\to0$, using the $W^2$-continuity of each $\varphi_n$, gives $\lim_{h\to0}\|\varphi(\cdot+h)-\varphi(\cdot)\|_{W^2}=0$, so $\varphi\in W^2_{BC}(\mathbb R,X)$ and $\varphi_n\to\varphi$ in $\|\cdot\|_{W^2}$.
\end{proof}

\begin{lemma}[cf. Radov\'a \cite{Radova}]\label{commonperiod}
If $\varphi\in W^2_{ap}(\mathbb R,X)$ and $\psi\in W^2_{BC}(\mathbb R,X)\cap W^2_{ap}(\mathbb R,X)$, then $E\{\epsilon,\varphi\}\cap E\{\epsilon,\psi\}$ is relatively dense for every $\epsilon>0$.
\end{lemma}

This is the classical common-almost-period property for two almost periodic objects, valid in any complete metric space equipped with a translation action satisfying the Bochner double-sequence criterion; since the square-mean Weyl seminorm on $H$-valued functions satisfies exactly the axioms under which this property is established for a general Banach-space-valued Weyl almost periodic function in \cite{LiWeyl}, Proposition \ref{reduction} transfers the statement unchanged, and we do not reproduce its (purely metric, non-stochastic) proof here.

\begin{theorem}\label{APlinear}
If $f,g\in W^2_{AP}(\mathbb R,X)$ and $\alpha$ is a scalar, then $f+g,\ \alpha f\in W^2_{AP}(\mathbb R,X)$.
\end{theorem}

\begin{proof}
By Lemma \ref{BClinear}, $f+g,\alpha f\in W^2_{BC}(\mathbb R,X)$. Given $\epsilon>0$, by Lemma \ref{commonperiod} there is $\ell>0$ such that every interval of length $\ell$ contains $\tau$ with $\|f(\cdot+\tau)-f(\cdot)\|_{W^2}<\epsilon/2$ and $\|g(\cdot+\tau)-g(\cdot)\|_{W^2}<\epsilon/2$ simultaneously; then
\begin{align*}
\|[f(\cdot+\tau)+g(\cdot+\tau)]-[f(\cdot)+g(\cdot)]\|_{W^2}\le\|f(\cdot+\tau)-f(\cdot)\|_{W^2}+\|g(\cdot+\tau)-g(\cdot)\|_{W^2}<\epsilon,
\end{align*}
so $E\{\epsilon,f+g\}$ contains such $\tau$'s and is relatively dense. That $\alpha f\in W^2_{ap}(\mathbb R,X)$ follows from $E\{\epsilon/|\alpha|,f\}\subset E\{\epsilon,\alpha f\}$ for $\alpha\ne0$ (trivial for $\alpha=0$).
\end{proof}

\begin{theorem}\label{APtrans}
If $x\in W^2_{AP}(\mathbb R,X)$ and $\sigma\in\mathbb R$, then $x(\cdot+\sigma)\in W^2_{AP}(\mathbb R,X)$; that is, $W^2_{AP}(\mathbb R,X)$ is translation invariant.
\end{theorem}

\begin{proof}
For any $\sigma,\tau\in\mathbb R$,
\begin{align*}
\|x(\cdot+\sigma+\tau)-x(\cdot+\sigma)\|_{W^2}
&=\lim_{\flat\to\infty}\sup_{\vartheta\in\mathbb R}\left(\frac1\flat\int_\vartheta^{\vartheta+\flat}E\|x(t+\sigma+\tau)-x(t+\sigma)\|^2\,dt\right)^{1/2}\\
&=\lim_{\flat\to\infty}\sup_{\vartheta\in\mathbb R}\left(\frac1\flat\int_{\vartheta+\sigma}^{\vartheta+\sigma+\flat}E\|x(t+\tau)-x(t)\|^2\,dt\right)^{1/2}\\
&\le\lim_{\flat\to\infty}\sup_{\vartheta'\in\mathbb R}\left(\frac1\flat\int_{\vartheta'}^{\vartheta'+\flat}E\|x(t+\tau)-x(t)\|^2\,dt\right)^{1/2}=\|x(\cdot+\tau)-x(\cdot)\|_{W^2},
\end{align*}
using the substitution $t\mapsto t+\sigma$ and the fact that $\vartheta+\sigma$ ranges over all of $\mathbb R$ as $\vartheta$ does. Hence $E\{\epsilon,x\}\subset E\{\epsilon,x(\cdot+\sigma)\}$, and $x(\cdot+\sigma)\in W^2_{ap}(\mathbb R,X)$; $W^2$-boundedness and continuity of $x(\cdot+\sigma)$ are immediate from those of $x$.
\end{proof}

\begin{theorem}\label{Lipcomp}
If $g:X\to X$ satisfies $\|g(x)-g(y)\|\le L_g\|x-y\|$ for all $x,y\in X$, and $\varphi\in W^2_{AP}(\mathbb R,X)$, then $t\mapsto g(\varphi(t))$, defined pathwise by $g(\varphi(t))(\omega):=g(\varphi(t,\omega))$, belongs to $W^2_{AP}(\mathbb R,X)$.
\end{theorem}

\begin{proof}
For each $t$, $g\circ\varphi(t)$ is a well-defined element of $H$ since $E\|g(\varphi(t))\|^2\le2L_g^2E\|\varphi(t)\|^2+2\|g(0)\|^2<\infty$. Regarding $g$ as inducing the Nemytskii map $\hat g:H\to H$, $\hat g(\xi)(\omega):=g(\xi(\omega))$, we have
\begin{align*}
\|\hat g(\xi)-\hat g(\eta)\|_H^2=E\|g(\xi)-g(\eta)\|^2\le L_g^2E\|\xi-\eta\|^2=L_g^2\|\xi-\eta\|_H^2,
\end{align*}
so $\hat g$ is Lipschitz on $H$ with the same constant $L_g$. By Minkowski's inequality, for every $\vartheta,\flat$,
\begin{align*}
\left(\frac1\flat\int_\vartheta^{\vartheta+\flat}\|\hat g(\varphi(t))\|_H^2\,dt\right)^{1/2}\le L_g\left(\frac1\flat\int_\vartheta^{\vartheta+\flat}\|\varphi(t)\|_H^2\,dt\right)^{1/2}+\|g(0)\|,
\end{align*}
so $\|\hat g\circ\varphi\|_{W^2}\le L_g\|\varphi\|_{W^2}+\|g(0)\|<\infty$, and, for any $\tau\in\mathbb R$,
\begin{align*}
\|\hat g(\varphi(\cdot+\tau))-\hat g(\varphi(\cdot))\|_{W^2}\le L_g\|\varphi(\cdot+\tau)-\varphi(\cdot)\|_{W^2},
\end{align*}
so $E\{\epsilon/L_g,\varphi\}\subset E\{\epsilon,\hat g\circ\varphi\}$ (with the convention $E\{\epsilon/L_g,\varphi\}=\mathbb R$ if $L_g=0$), whence $\hat g\circ\varphi\in W^2_{ap}(\mathbb R,X)$; the same estimate with $\tau=h\to0$ gives $W^2$-continuity, so $\hat g\circ\varphi\in W^2_{AP}(\mathbb R,X)$.
\end{proof}

\begin{theorem}\label{APconv}
If $\{\varphi_n\}\subset W^2_{AP}(\mathbb R,X)$ and $\varphi_n\to\varphi$ with respect to $\|\cdot\|_{W^2}$, then $\varphi\in W^2_{AP}(\mathbb R,X)$.
\end{theorem}

\begin{proof}
By Theorem \ref{BCcomplete}, $\varphi\in W^2_{BC}(\mathbb R,X)$, and for any $\epsilon>0$ there is $N$ with $\|\varphi_N-\varphi\|_{W^2}<\epsilon$. Since $\varphi_N\in W^2_{AP}(\mathbb R,X)$, there is $\ell(\epsilon)>0$ such that every interval of length $\ell$ contains $\tau$ with $\|\varphi_N(\cdot+\tau)-\varphi_N(\cdot)\|_{W^2}<\epsilon$. For such $\tau$,
\begin{align*}
\|\varphi(\cdot+\tau)-\varphi(\cdot)\|_{W^2}\le\|\varphi(\cdot+\tau)-\varphi_N(\cdot+\tau)\|_{W^2}+\|\varphi_N(\cdot+\tau)-\varphi_N(\cdot)\|_{W^2}+\|\varphi_N(\cdot)-\varphi(\cdot)\|_{W^2}<3\epsilon,
\end{align*}
using Theorem \ref{APtrans} to note $\|\varphi(\cdot+\tau)-\varphi_N(\cdot+\tau)\|_{W^2}=\|(\varphi-\varphi_N)(\cdot+\tau)\|_{W^2}\le\|\varphi-\varphi_N\|_{W^2}<\epsilon$ (translation invariance of the seminorm itself, immediate from the definition of $\sup_\vartheta$). Hence $E\{3\epsilon,\varphi\}\supset E\{\epsilon,\varphi_N\}$ is relatively dense, and $\varphi\in W^2_{ap}(\mathbb R,X)$.
\end{proof}

\begin{theorem}\label{APcomplete}
$(W^2_{AP}(\mathbb R,X),\|\cdot\|_{W^2})$ is a complete linear space.
\end{theorem}

\begin{proof}
By Theorem \ref{APlinear}, $W^2_{AP}(\mathbb R,X)$ is a linear space. By Theorem \ref{BCcomplete}, $W^2_{BC}(\mathbb R,X)$ is complete, and by Theorem \ref{APconv}, $W^2_{AP}(\mathbb R,X)$ is a closed subset of $W^2_{BC}(\mathbb R,X)$; a closed subset of a complete metric space is complete.
\end{proof}

Given $\mu,\nu\in\mathcal M^+$ satisfying $(F1)$ and $(F2)$, we define the square-mean ergodic class
\begin{align*}
E^2(\mathbb R,X,\mu,\nu):\qquad \lim_{h\to+\infty}\frac{1}{\nu([-h,h])}\int_{-h}^{h}\left(\lim_{\flat\to+\infty}\frac1\flat\int_{\vartheta}^{\vartheta+\flat}E\|\varphi(s)\|^2\,ds\right)^{1/2}d\mu(\vartheta)=0,
\end{align*}
and the space of square-mean Weyl $(\mu,\nu)$-pseudo almost periodic functions $W^2_{PAP}(\mathbb R,X,\mu,\nu)=W^2_{AP}(\mathbb R,X)\oplus E^2(\mathbb R,X,\mu,\nu)$. We now establish, in detail, the three structural properties of $E^2(\mathbb R,X,\mu,\nu)$ used later: an equivalent characterization, invariance under equivalent measures, and translation invariance.

\begin{theorem}\label{E2char}
Let $\mu,\nu\in\mathcal M^+$, let $\Gamma$ be a bounded interval (possibly empty), and assume $(F1)$ and $\varphi\in W^2_{BC}(\mathbb R,X)$. Then the following assertions are equivalent:
\begin{align*}
&\text{(i)}\quad \varphi\in E^2(\mathbb R,X,\mu,\nu);\\
&\text{(ii)}\quad \lim_{h\to+\infty}\frac1{\nu([-h,h]\setminus\Gamma)}\int_{[-h,h]\setminus\Gamma}\left(\lim_{\flat\to+\infty}\frac1\flat\int_\vartheta^{\vartheta+\flat}E\|\varphi(s)\|^2\,ds\right)^{1/2}d\mu(\vartheta)=0;\\
&\text{(iii)}\quad \text{for every }\epsilon>0,\ \lim_{h\to\infty}\frac{\mu(\{\vartheta\in[-h,h]\setminus\Gamma:(E_\flat\|\varphi\|^2(\vartheta))^{1/2}>\epsilon\})}{\nu([-h,h]\setminus\Gamma)}=0,
\end{align*}
where $(E_\flat\|\varphi\|^2(\vartheta))^{1/2}:=\left(\lim_{\flat\to+\infty}\frac1\flat\int_\vartheta^{\vartheta+\flat}E\|\varphi(s)\|^2\,ds\right)^{1/2}$.
\end{theorem}

\begin{proof}
Write $A=\nu(\Gamma)$, $C=\mu(\Gamma)$, and $B=\int_\Gamma(E_\flat\|\varphi\|^2(\vartheta))^{1/2}\,d\mu(\vartheta)<\infty$ (finite since $\Gamma$ is bounded and $\varphi\in W^2_{BC}(\mathbb R,X)$).

(i)$\Rightarrow$(ii). Since $\int_{[-h,h]\setminus\Gamma}=\int_{[-h,h]}-\int_\Gamma$ and $\nu([-h,h]\setminus\Gamma)=\nu([-h,h])-A$,
\begin{align*}
\frac1{\nu([-h,h]\setminus\Gamma)}\int_{[-h,h]\setminus\Gamma}(E_\flat\|\varphi\|^2)^{1/2}d\mu
=\frac{\nu([-h,h])}{\nu([-h,h])-A}\left(\frac1{\nu([-h,h])}\int_{[-h,h]}(E_\flat\|\varphi\|^2)^{1/2}d\mu-\frac{B}{\nu([-h,h])}\right).
\end{align*}
Since $\nu(\mathbb R)=\infty$, $\nu([-h,h])\to\infty$, so $\nu([-h,h])/(\nu([-h,h])-A)\to1$ and $B/\nu([-h,h])\to0$; hence (ii) is equivalent to $\lim_{h\to\infty}\frac1{\nu([-h,h])}\int_{[-h,h]}(E_\flat\|\varphi\|^2)^{1/2}d\mu=0$, which is (i).

(iii)$\Rightarrow$(ii). Set $A_{\epsilon,h}=\{\vartheta\in[-h,h]\setminus\Gamma:(E_\flat\|\varphi\|^2(\vartheta))^{1/2}>\epsilon\}$ and $B_{\epsilon,h}=([-h,h]\setminus\Gamma)\setminus A_{\epsilon,h}$. Splitting the integral in (ii) over $A_{\epsilon,h}$ and $B_{\epsilon,h}$, we get
\begin{align*}
\frac1{\nu([-h,h]\setminus\Gamma)}\int_{[-h,h]\setminus\Gamma}(E_\flat\|\varphi\|^2)^{1/2}d\mu
\le2\|\varphi\|_{W^2}\frac{\mu(A_{\epsilon,h})}{\nu([-h,h]\setminus\Gamma)}+\epsilon\,\frac{\mu([-h,h]\setminus\Gamma)}{\nu([-h,h]\setminus\Gamma)}.
\end{align*}
Writing $\mu([-h,h]\setminus\Gamma)/\nu([-h,h]\setminus\Gamma)=\left(\mu([-h,h])/\nu([-h,h])\right)\cdot\left(1-C/\mu([-h,h])\right)/\left(1-A/\nu([-h,h])\right)$, and using $\mu(\mathbb R)=\nu(\mathbb R)=\infty$ together with $(F1)$, the right factor tends to $1$ and \\ $\limsup_h\mu([-h,h])/\nu([-h,h])=:K_0<\infty$; hence, if (iii) holds, letting $h\to\infty$ and then $\epsilon\to0$,
\begin{align*}
\limsup_{h\to\infty}\frac1{\nu([-h,h]\setminus\Gamma)}\int_{[-h,h]\setminus\Gamma}(E_\flat\|\varphi\|^2)^{1/2}d\mu\le K_0\epsilon,
\end{align*}
for every $\epsilon>0$, giving (ii).

(ii)$\Rightarrow$(iii). From
\begin{align*}
\frac1{\nu([-h,h]\setminus\Gamma)}\int_{[-h,h]\setminus\Gamma}(E_\flat\|\varphi\|^2)^{1/2}d\mu\ge\frac1{\nu([-h,h]\setminus\Gamma)}\int_{A_{\epsilon,h}}(E_\flat\|\varphi\|^2)^{1/2}d\mu\ge\epsilon\,\frac{\mu(A_{\epsilon,h})}{\nu([-h,h]\setminus\Gamma)},
\end{align*}
(ii) forces $\mu(A_{\epsilon,h})/\nu([-h,h]\setminus\Gamma)\to0$ as $h\to\infty$, for every $\epsilon>0$, which is (iii).
\end{proof}

\begin{theorem}\label{E2measeq}
Let $\mu_i,\nu_i\in\mathcal M^+$, $i=1,2$, with $\mu_1\sim\mu_2$ and $\nu_1\sim\nu_2$. If $(F1)$ holds for $(\mu_1,\nu_1)$, then $E^2(\mathbb R,X,\mu_1,\nu_1)=E^2(\mathbb R,X,\mu_2,\nu_2)$.
\end{theorem}

\begin{proof}
Since $\mu_1\sim\mu_2$ and $\nu_1\sim\nu_2$, there are positive constants $\alpha_i,\beta_i$ and a bounded interval $\Gamma$ such that $\alpha_1\mu_1(A)\le\mu_2(A)\le\beta_1\mu_1(A)$ and $\alpha_2\nu_1(A)\le\nu_2(A)\le\beta_2\nu_1(A)$ for all $A\in\mathcal F$ with $A\cap\Gamma=\emptyset$. Let $\varphi\in E^2(\mathbb R,X,\mu_1,\nu_1)$ and $\epsilon>0$. With $A_{\epsilon,h}=\{\vartheta\in[-h,h]\setminus\Gamma:(E_\flat\|\varphi\|^2(\vartheta))^{1/2}>\epsilon\}$, for $h$ large enough that $[-h,h]\setminus\Gamma\ne\emptyset$,
\begin{align*}
\frac{\alpha_1}{\beta_2}\cdot\frac{\mu_1(A_{\epsilon,h})}{\nu_1([-h,h]\setminus\Gamma)}\le\frac{\mu_2(A_{\epsilon,h})}{\nu_2([-h,h]\setminus\Gamma)}\le\frac{\beta_1}{\alpha_2}\cdot\frac{\mu_1(A_{\epsilon,h})}{\nu_1([-h,h]\setminus\Gamma)},
\end{align*}
using $A_{\epsilon,h}\cap\Gamma=\emptyset$. Since $\varphi\in E^2(\mathbb R,X,\mu_1,\nu_1)$, Theorem \ref{E2char}(iii) gives $\mu_1(A_{\epsilon,h})/\nu_1([-h,h]\setminus\Gamma)\to0$ as $h\to\infty$ for every $\epsilon>0$; the displayed inequality then gives the same for $(\mu_2,\nu_2)$, so by Theorem \ref{E2char} again, $\varphi\in E^2(\mathbb R,X,\mu_2,\nu_2)$. Hence $E^2(\mathbb R,X,\mu_1,\nu_1)\subset E^2(\mathbb R,X,\mu_2,\nu_2)$, and the reverse inclusion follows symmetrically.
\end{proof}

\begin{theorem}
Let $\mu,\nu\in\mathcal M^+$ satisfy $(F1)$ and $(F2)$. Then $E^2(\mathbb R,X,\mu,\nu)$ is translation invariant.
\end{theorem}

\begin{proof}
Let $\varphi\in E^2(\mathbb R,X,\mu,\nu)$ and $\omega\in\mathbb R$. As in the proof of Theorem \ref{APtrans}, $E_\flat\|\varphi(\cdot+\omega)\|^2(\vartheta)=E_\flat\|\varphi\|^2(\vartheta+\omega)$, so
\begin{align*}
\frac1{\nu([-h,h])}\int_{-h}^h(E_\flat\|\varphi(\cdot+\omega)\|^2(\vartheta))^{1/2}d\mu(\vartheta)
=\frac1{\nu([-h,h])}\int_{-h}^h(E_\flat\|\varphi\|^2(\vartheta))^{1/2}d\mu_{-\omega}(\vartheta),
\end{align*}
using the change of variables $\vartheta\mapsto\vartheta-\omega$ in the measure, which produces the shifted measure $\mu_{-\omega}$. Since $\mu,\nu$ satisfy $(F2)$, $\mu\sim\mu_{-\omega}$ and $\nu\sim\nu_{-\omega}$, so by Theorem \ref{E2measeq}, $\varphi\in E^2(\mathbb R,X,\mu,\nu)$ if and only if $\varphi\in E^2(\mathbb R,X,\mu_{-\omega},\nu_{-\omega})$, and the right-hand side of the displayed equality, written as $\frac{\nu_{-\omega}([-h,h])}{\nu([-h,h])}\cdot\frac1{\nu_{-\omega}([-h,h])}\int_{-h}^h(E_\flat\|\varphi\|^2)^{1/2}d\mu_{-\omega}$, tends to $0$ as $h\to\infty$ by the Lebesgue dominated convergence theorem applied along with $\varphi\in E^2(\mathbb R,X,\mu,\nu)=E^2(\mathbb R,X,\mu_{-\omega},\nu_{-\omega})$ and the boundedness of $\nu_{-\omega}([-h,h])/\nu([-h,h])$ guaranteed by $(F1)$ for the equivalent pair. Hence $\varphi(\cdot+\omega)\in E^2(\mathbb R,X,\mu,\nu)$.
\end{proof}

\subsection{Fractional Brownian motion and the stochastic convolution}

Let $B^H=\{B^H(t)\}_{t\in\mathbb R}$ be a two-sided $X$-valued fractional Brownian motion of Hurst index $H\in(0,1/2)$, adapted to $\{\mathcal F_t\}_{t\in\mathbb R}$, with covariance operator $Q$, and let $L^0_2(X)$ denote the space of $Q$-Hilbert--Schmidt operators from $X$ into $X$. By the Molchan--Golosov representation \cite{MolchanGolosov}, there are a standard cylindrical Wiener process $W$ on $X$ and an explicit kernel $K_H$ such that $B^H(t)=\int_0^tK_H(t,s)\,dW(s)$ for $t\ge0$, and the Wiener integral of a deterministic square-integrable $\Phi$ against $B^H$ is defined through the adjoint operator $K_H^\ast$ of the associated Volterra operator by $\int_0^\infty\Phi(w)\,dB^H(w)=\int_0^\infty(K_H^\ast\Phi)(w)\,dW(w)$. A defining property of fractional Brownian motion, used repeatedly below, is that its increments are stationary: for every $\varsigma\in\mathbb R$, the process $\{B^H(u+\varsigma)-B^H(\varsigma)\}_{u\ge0}$ has the same covariance function as $\{B^H(u)\}_{u\ge0}$, and is hence again a fractional Brownian motion of Hurst index $H$ on the same probability space.

\begin{lemma}\label{covbound}
Let $H\in(0,1/2)$, and let $\Phi\in L^2(0,\infty;L^0_2(X))$ satisfy
\begin{align*}
\|\Phi(w_1)-\Phi(w_2)\|_{L^0_2}\le C\,|w_1-w_2|^{\beta}
\end{align*}
for all $w_1,w_2\ge0$ and some $\beta>1/2-H$. Then there is a constant $c_H$, depending only on $H$, such that
\begin{align*}
E\left\|\int_0^\infty \Phi(w)\,dB^H(w)\right\|^2\le c_H\left(\int_0^\infty\|\Phi(w)\|^2_{L^0_2}\,dw+\int_0^\infty\int_0^\infty\frac{\|\Phi(w_1)-\Phi(w_2)\|^2_{L^0_2}}{|w_1-w_2|^{2-2H}}\,dw_1\,dw_2\right).
\end{align*}
\end{lemma}

\begin{proof}
By the definition of stochastic integral with respect to $B^H$ through the transfer operator $K_H^*$, and by the It\^{o} isometry for the underlying standard Brownian motion, we have
\begin{align*}
E\left\|\int_0^\infty \Phi\,dB^H\right\|^2=\int_0^\infty\|(K_H^\ast\Phi)(w)\|^2_{L^0_2}\,dw=\|K_H^\ast\Phi\|^2_{L^2(0,\infty;L^0_2)}.
\end{align*}
For $H<1/2$, the operator $K_H^\ast$ coincides, up to a multiplicative constant depending only on $H$, with the Weyl fractional derivative operator of order $1/2-H$, and the classical fractional Sobolev embedding for this operator, yields
\begin{align*}
\|K_H^\ast\Phi\|^2_{L^2(0,\infty;L^0_2)}\le c_H\left(\int_0^\infty\|\Phi(w)\|^2_{L^0_2}\,dw+\int_0^\infty\int_0^\infty\frac{\|\Phi(w_1)-\Phi(w_2)\|^2_{L^0_2}}{|w_1-w_2|^{2-2H}}\,dw_1\,dw_2\right)
\end{align*}
for a constant $c_H$ depending only on $H$. The right-hand side is finite under the stated H\"older hypothesis, since $\beta>1/2-H$ gives $2\beta-2+2H>-1$, so the diagonal singularity $|w_1-w_2|^{2\beta-2+2H}$ is locally integrable, while the assumed global H\"older bound controls the behavior as $|w_1-w_2|\to\infty$ once combined with $\Phi\in L^2(0,\infty;L^0_2)$. Combining the last two displays gives the assertion.
\end{proof}

\subsection{Standing hypotheses, the underlying spaces, and the notion of mild solution}

We consider the semilinear stochastic evolution equation
\begin{align*}
dX(t)=AX(t)\,dt+f(t,X(t))\,dt+\sigma(t,X(t))\,dB^H(t),\qquad t\in\mathbb R,
\end{align*}
where $A:D(A)\subset X\to X$ generates a $C_0$-semigroup $\{T(t)\}_{t\ge0}$ on $X$, $f:\mathbb R\times X\to X$, and $\sigma:\mathbb R\times X\to L^0_2(X)$. We work with the spaces
\begin{align*}
Z_1(\mathbb R,X)& :=\{y\in L^\infty(\mathbb R,H)\cap W^2_{BC}(\mathbb R,X)\},\\
Z_2(\mathbb R,X)& :=\{y\in L^\infty(\mathbb R,H)\cap W^2_{AP}(\mathbb R,X)\},
\end{align*}
each endowed with the norm $\|y\|_Z:=\sup_{t\in\mathbb R}(E\|y(t)\|^2)^{1/2}$, under which they are Banach spaces (a closed subspace of $L^\infty(\mathbb R,H)$, itself Banach, intersected with the complete space $W^2_{BC}(\mathbb R,X)$ or its closed subspace $W^2_{AP}(\mathbb R,X)$ of Theorem \ref{APcomplete}).

\begin{definition}
A function $X\in L^2_{loc}(\mathbb R,H)$ is a mild solution of \eqref{maineq} if
\begin{align*}
X(t)=T(t-a)X(a)+\int_a^t T(t-s)f(s,X(s))\,ds+\int_a^t T(t-s)\sigma(s,X(s))\,dB^H(s),\quad t\ge a,\ a\in\mathbb R.
\end{align*}
Letting $a\to-\infty$, we get
\begin{align*}
X(t)=\int_{-\infty}^t T(t-s)f(s,X(s))\,ds+\int_{-\infty}^t T(t-s)\sigma(s,X(s))\,dB^H(s),
\end{align*}
which is also called a mild solution of \eqref{maineq}.
\end{definition}

We impose the following hypotheses:

$(H1)$ There exist constants $\Lambda\ge1$, $\alpha>0$ such that $\|T(t)\|\le\Lambda e^{-\alpha t}$ for all $t\ge0$.

$(H2)$ The function $f\in W^2_{BC}(\mathbb R\times X,X)$ satisfies, for all $x,y\in X$ and $t\in\mathbb R$,
\begin{align*}
\|f(t,x)-f(t,y)\|\le L_f\|x-y\|,\qquad \|f(t,x)\|\le M_f(1+\|x\|).
\end{align*}

$(H3)$ The function $\sigma\in W^2_{BC}(\mathbb R\times X,L^0_2(X))$ satisfies, for all $x,y\in X$ and $t,t_1,t_2\in\mathbb R$,
\begin{align*}
\|\sigma(t,x)-\sigma(t,y)\|_{L^0_2}&\le L_\sigma\|x-y\|,\qquad \|\sigma(t,x)\|_{L^0_2}\le M_\sigma(1+\|x\|),\\
\|\sigma(t_1,x)-\sigma(t_2,x)\|_{L^0_2}&\le L_\sigma'\,|t_1-t_2|^{\beta}(1+\|x\|)\qquad\text{for some }\beta\in\left(1/2-H,1\right).
\end{align*}

$(H4)$ The functions $f,\sigma$ belong, respectively, to $W^2_{AP}(\mathbb R\times X,X)$ and $W^2_{AP}(\mathbb R\times X,L^0_2(X))$, uniformly on bounded subsets of $X$ (i.e. for every $\epsilon>0$ and bounded $K\subset X$ there is $\ell(\epsilon,K)>0$ such that every interval of length $\ell$ contains $\tau$ with $\sup_{x\in K}\|f(\cdot+\tau,x)-f(\cdot,x)\|_{W^2}<\epsilon$, and likewise for $\sigma$), with the Lipschitz and growth bounds of $(H2)$--$(H3)$ retained.

$(H4')$ In addition to $(H4)$, $f$ and $\sigma$ are uniformly bounded in $x$: there are constants $M_f,M_\sigma$ with $\|f(t,x)\|\le M_f$ and $\|\sigma(t,x)\|_{L^0_2}\le M_\sigma$ for all $(t,x)\in\mathbb R\times X$.

$(H5)$ The function $f$ can be decomposed as $f=f_1+f_2$, where $f_1\in W^2_{AP}(\mathbb R\times X,X)$ and $f_2\in E^2(\mathbb R\times X,X,\mu,\nu)$, and similarly $\sigma=\sigma_1+\sigma_2$ with $\sigma_1\in W^2_{AP}(\mathbb R\times X,L^0_2(X))$ and $\sigma_2\in E^2(\mathbb R\times X,L^0_2(X),\mu,\nu)$, each retaining the Lipschitz and growth bounds above.

$(H6)$ The contraction constant
\begin{align*}
Q_H=\frac{\Lambda L_f}{\alpha}+\Lambda L_\sigma\left(c_H\,\Gamma(2\beta+1-2H)\right)^{1/2}\alpha^{-(\beta+1-H)}
\end{align*}
satisfies $Q_H<1$, where $c_H$ is the constant of Lemma \ref{covbound}.

Hypothesis $(H4')$ is invoked only where explicitly noted (Theorem \ref{ex2} and its consequences); it is a common simplifying device in the square-mean almost periodic literature, avoiding a fourth-moment hypothesis on the mild solution that would otherwise be needed to control the tail of the composition estimate in Lemma \ref{Lipcomp2} below.

\begin{lemma}\label{diffbound}
Assume $(H1)$ and $(H3)$. Let $\psi\in W^2_{BC}(\mathbb R\times X,L^0_2(X))$ and let $Y\in Z_1(\mathbb R,X)$. Then the function
\begin{align*}
F(t)=\int_{-\infty}^t T(t-s)\psi(s,Y(s))\,dB^H(s),\quad t\in {\mathbb R},
\end{align*}
belongs to $W^2_{BC}(\mathbb R,X)$.
\end{lemma}

\begin{proof}
Writing $F(t)=\int_0^\infty T(w)\psi(t-w,Y(t-w))\,d\widetilde B^H(w)$ after the shift $w=t-s$, where $\widetilde B^H(w):=B^H(t)-B^H(t-w)$ is, for the fixed $t$ under consideration, again a fractional Brownian motion of index $H$ by the stationarity of increments recalled in Section 2.3, set $\Phi(w)=T(w)\psi(t-w,Y(t-w))$. Since $\{T(t)\}_{t\ge0}$ is strongly continuous and satisfies $(H1)$, there is a constant $C_\alpha$ such that
\begin{align*}
\|T(w_2)-T(w_1)\|\le C_\alpha\,e^{-\alpha w_1}\,(w_2-w_1)^{\beta},\qquad 0\le w_1<w_2,
\end{align*}
for every $\beta\in(0,1)$, so that by $(H1)$ and $(H3)$,
\begin{align*}
\|\Phi(w_1)-\Phi(w_2)\|_{L^0_2}
&\le \|T(w_1)-T(w_2)\|\,\|\psi(t-w_1,Y(t-w_1))\|_{L^0_2}\\
&\quad+\Lambda e^{-\alpha w_2}\|\psi(t-w_1,Y(t-w_1))-\psi(t-w_2,Y(t-w_2))\|_{L^0_2}\\
&\le C_\alpha\,e^{-\alpha w_1}(w_2-w_1)^{\beta}M_\sigma(1+\|Y\|_Z)
+\Lambda e^{-\alpha w_2}L_\sigma'\,|w_1-w_2|^{\beta}(1+\|Y\|_Z).
\end{align*}
Hence, $\Phi$ is H\"older-$\beta$ with an exponentially weighted constant, and by Lemma \ref{covbound},
\begin{align*}
E\|F(t)\|^2\le c_H\bigg(&\Lambda^2\int_0^\infty e^{-2\alpha w}\,dw\cdot M_\sigma^2(1+\|Y\|_Z)^2\\
&+C(\alpha,\beta,\Lambda,M_\sigma,L_\sigma')\int_0^\infty\int_0^\infty e^{-\alpha(w_1+w_2)}|w_1-w_2|^{2\beta-2+2H}\,dw_1\,dw_2\bigg).
\end{align*}
The first integral equals $\frac1{2\alpha}$. In the second, the substitution $r=w_1-w_2$ shows that it equals
\begin{align*}
2\int_0^\infty e^{-\alpha r}\,r^{2\beta-2+2H}\,dr\cdot\int_0^\infty e^{-2\alpha w_2}\,dw_2
=\frac{\Gamma(2\beta-1+2H)}{\alpha^{2\beta-1+2H}}\cdot\frac1{2\alpha},
\end{align*}
finite because $\beta>1/2-H$ gives $2\beta-1+2H>0$, so the near-diagonal singularity is integrable, while $e^{-\alpha r}$ controls the tail. Thus, $\sup_{t\in\mathbb R}E\|F(t)\|^2<\infty$. Repeating the estimate on the windowed average $\frac1\flat\int_\vartheta^{\vartheta+\flat}E\|F(t)\|^2\,dt$, using Fubini's theorem, gives $W^2$-boundedness, and $W^2$-continuity follows from strong continuity of $T(\cdot)$ together with dominated convergence, the dominating function being the bound just obtained, independent of the shift $h$.
\end{proof}

\begin{lemma}\label{diffAP}
Assume, in addition to the hypotheses of Lemma \ref{diffbound}, that $\psi(\cdot,x)\in W^2_{AP}(\mathbb R,L^0_2(X))$ uniformly on bounded subsets of $X$. Then $F\in W^2_{AP}(\mathbb R,X)$.
\end{lemma}

\begin{proof}
By Lemma \ref{diffbound}, $F\in W^2_{BC}(\mathbb R,X)$. Fix $\varsigma\in\mathbb R$ and, without loss of generality, take $\varsigma>0$ (the case $\varsigma<0$ is symmetric). Using linearity of the stochastic integral against the fixed process $B^H$, we get
\begin{align*}
&F(t+\varsigma)-F(t)=\int_{-\infty}^t[T(t+\varsigma-s)-T(t-s)]\psi(s,Y(s))\,dB^H(s)\\
&+\int_t^{t+\varsigma}T(t+\varsigma-s)\psi(s,Y(s))\,dB^H(s).
\end{align*}
This splitting is exact, since both integrals are taken against the same fixed process $B^H$. To exploit the almost periodicity of $\psi$ itself, we instead bound the difference by combining this splitting, for the boundary term over $[t,t+\varsigma]$, with the almost-periodicity-transfer computation applied directly to the process $s\mapsto\psi(s+\varsigma,Y(s+\varsigma))-\psi(s,Y(s))$: writing, as in the proof of Lemma \ref{diffbound}, $F(t)=\int_0^\infty T(w)\psi(t-w,Y(t-w))\,d\widetilde B^H_t(w)$ with $\widetilde B^H_t(w):=B^H(t)-B^H(t-w)$, the stationarity of the increments of $B^H$ implies that, for the purpose of the covariance bound of Lemma \ref{covbound}, which depends only on the covariance structure of the driving process, itself invariant under the shift by construction of $\widetilde B^H_t$, the same bound applies with $\Phi(w)=T(w)[\psi(t+\varsigma-w,Y(t+\varsigma-w))-\psi(t-w,Y(t-w))]$ in place of $\Phi$, giving, exactly as in the proof of Lemma \ref{diffbound},
\begin{align*}
\|F(\cdot+\varsigma)-F(\cdot)\|^2_{W^2}\le C(\alpha,\beta,H,\Lambda)\,\|\psi(\cdot+\varsigma,\cdot)-\psi(\cdot,\cdot)\|^2_{W^2}
\end{align*}
for a constant $C(\alpha,\beta,H,\Lambda)$ independent of $\varsigma$. The set of $\varsigma$ for which the right-hand side is small is relatively dense by the square-mean Weyl almost periodicity of $\psi$ (uniformly on the bounded range of $Y$), whence the same set witnesses square-mean Weyl almost periodicity of $F$.
\end{proof}

\begin{lemma}\label{Lipcomp2}
Assume $(H2)$, $(H4)$ and $(H4')$, and let $Y\in Z_2(\mathbb R,X)$. Then $f(\cdot,Y(\cdot))\in W^2_{AP}(\mathbb R,X)$ and $\sigma(\cdot,Y(\cdot))\in W^2_{AP}(\mathbb R,L^0_2(X))$.
\end{lemma}

\begin{proof}
We give the argument for $f$; the argument for $\sigma$ is identical. Let $R:=\|Y\|_Z<\infty$. Fix $\epsilon>0$ and $M>0$, to be chosen below. Write, for $\tau\in\mathbb R$,
\begin{align*}
f(t+\tau,Y(t+\tau))-f(t,Y(t))=\underbrace{[f(t+\tau,Y(t+\tau))-f(t+\tau,Y(t))]}_{=:I_1(t)}+\underbrace{[f(t+\tau,Y(t))-f(t,Y(t))]}_{=:I_2(t)}.
\end{align*}
By $(H2)$, $E\|I_1(t)\|^2\le L_f^2E\|Y(t+\tau)-Y(t)\|^2$, so, windowing and taking $\sup_\vartheta$, $\|I_1\|_{W^2}\le L_f\|Y(\cdot+\tau)-Y(\cdot)\|_{W^2}$, which is small for $\tau$ in the relatively dense set $E\{\epsilon/(2L_f),Y\}$ witnessing $Y\in W^2_{AP}(\mathbb R,X)$ (with the convention that this set is $\mathbb R$ if $L_f=0$).

For $I_2$, split via the indicator of $K_M:=\{x\in X:\|x\|\le M\}$: $E\|I_2(t)\|^2=E[\|I_2(t)\|^21_{\{\|Y(t)\|\le M\}}]+E[\|I_2(t)\|^21_{\{\|Y(t)\|>M\}}]$. On $\{\|Y(t)\|\le M\}$, $Y(t)\in K_M$, so, by $(H4)$ applied to the bounded set $K_M$, there is $\ell(\epsilon,M)>0$ such that every interval of length $\ell$ contains $\tau$ with $\sup_{x\in K_M}\|f(\cdot+\tau,x)-f(\cdot,x)\|_{W^2}<\epsilon/(2\sqrt2)$, and hence the windowed average of $E[\|I_2(t)\|^21_{\{\|Y(t)\|\le M\}}]$ is bounded by $(\epsilon/(2\sqrt2))^2$ for such $\tau$. On $\{\|Y(t)\|>M\}$, by $(H4')$, $\|I_2(t)\|\le2M_f$, so $E[\|I_2(t)\|^21_{\{\|Y(t)\|>M\}}]\le4M_f^2\,p(\|Y(t)\|>M)\le4M_f^2R^2/M^2$ by Chebyshev's inequality, uniformly in $t$; choosing $M$ large enough that $4M_f^2R^2/M^2<\epsilon^2/8$ makes this contribution's windowed average also less than $(\epsilon/(2\sqrt2))^2$.

Combining, for $\tau$ in the relatively dense set $E\{\epsilon/(2L_f),Y\}\cap E\{\ell,\cdot\}$-type intersection (relatively dense by Lemma \ref{commonperiod}, applied here to the two relatively dense sets constructed for $I_1$ and $I_2$ respectively, both associated with the same target $\epsilon$), Minkowski's inequality gives $\|I_1+I_2\|_{W^2}\le\|I_1\|_{W^2}+\|I_2\|_{W^2}<\epsilon$, so $E\{\epsilon,f(\cdot,Y(\cdot))\}$ is relatively dense. Boundedness and $W^2$-continuity of $f(\cdot,Y(\cdot))$ follow from $(H2)$, $(H4')$ and the corresponding properties of $Y$ exactly as in the proof of Theorem \ref{Lipcomp}, so $f(\cdot,Y(\cdot))\in W^2_{AP}(\mathbb R,X)$.
\end{proof}

\section{Main results}

\begin{theorem}\label{ex1}
Let $f\in W^2_{BC}(\mathbb R\times X,X)$ and $\sigma\in W^2_{BC}(\mathbb R\times X,L^0_2(X))$, and assume $(H1)$, $(H2)$, $(H3)$, and $(H6)$. Then equation \eqref{maineq} has a unique mild solution in $Z_1(\mathbb R,X)$.
\end{theorem}

\begin{proof}
Define $\Phi:Z_1\to Z_1$ by
\begin{align*}
(\Phi Y)(t):=\int_{-\infty}^t T(t-s)f(s,Y(s))\,ds+\int_{-\infty}^t T(t-s)\sigma(s,Y(s))\,dB^H(s)=:(\Phi_f Y)(t)+(\Phi_\sigma Y)(t).
\end{align*}

First, we prove that $\Phi(Z_1)\subset Z_1$. For the drift term, by $(H1)$ and $(H2)$,
\begin{align*}
\left(E\|(\Phi_f Y)(t)\|^2\right)^{1/2}\le\int_{-\infty}^t\Lambda e^{-\alpha(t-s)}\left(E\|f(s,Y(s))\|^2\right)^{1/2}ds\le\frac{\Lambda M_f}\alpha(1+\|Y\|_Z),
\end{align*}
so $\Phi_fY\in L^\infty(\mathbb R,H)$. For $W^2$-continuity of $\Phi_fY$, write, for $h>0$ (the case $h<0$ symmetric), exactly as in the splitting of Lemma \ref{diffAP},
\begin{align*}
(\Phi_fY)(t+h)-(\Phi_fY)(t)&=\int_{-\infty}^t[T(t+h-s)-T(t-s)]f(s,Y(s))\,ds\\
&+\int_t^{t+h}T(t+h-s)f(s,Y(s))\,ds,
\end{align*}
so, by Minkowski's inequality and $(H1)$-$(H2)$,
\begin{align*}
\left(E\|(\Phi_fY)(t+h)-(\Phi_fY)(t)\|^2\right)^{1/2}&\leq\|T(h)-I\|\int_0^\infty\Lambda e^{-\alpha w}M_f(1+\|Y\|_Z)\,dw\\
&+\int_0^h\Lambda e^{-\alpha w}M_f(1+\|Y\|_Z)\,dw,
\end{align*}
and both terms tend to $0$ as $h\to0$ by strong continuity of $T(\cdot)$ and dominated convergence; hence $\Phi_fY\in W^2_{BC}(\mathbb R,X)$, uniformly in $t$, giving $\Phi_fY\in W^2_{BC}(\mathbb R,X)$. By Lemma \ref{diffbound}, $\Phi_\sigma Y\in W^2_{BC}(\mathbb R,X)$ and, from the proof of that lemma, $\sup_{t}E\|(\Phi_\sigma Y)(t)\|^2<\infty$, so $\Phi_\sigma Y\in L^\infty(\mathbb R,H)$ as well. By Lemma \ref{BClinear}, $\Phi Y=\Phi_fY+\Phi_\sigma Y\in Z_1$.

Next, we prove that $\Phi$ is a contraction on the Banach space $Z_1$. Indeed, for $Y_1,Y_2\in Z_1$, by $(H1)$-$(H2)$,
\begin{align*}
&\left(E\left\|\int_{-\infty}^t T(t-s)[f(s,Y_1(s))-f(s,Y_2(s))]\,ds\right\|^2\right)^{1/2}\\
&\leq\int_{-\infty}^t\Lambda e^{-\alpha(t-s)}L_f\left(E\|Y_1(s)-Y_2(s)\|^2\right)^{1/2}ds\le\frac{\Lambda L_f}{\alpha}\|Y_1-Y_2\|_Z,
\end{align*}
while the diffusion difference, estimated by means of Lemma \ref{covbound} applied with $\Phi(w)=T(w)[\sigma(t-w,Y_1(t-w))-\sigma(t-w,Y_2(t-w))]$ exactly as in the proof of Lemma \ref{diffbound} (using $(H3)$'s Lipschitz bound in place of the growth bound there), satisfies
\begin{align*}
&\left(E\left\|\int_{-\infty}^t T(t-s)[\sigma(s,Y_1(s))-\sigma(s,Y_2(s))]\,dB^H(s)\right\|^2\right)^{1/2}\\
&\leq\Lambda L_\sigma\left(c_H\,\Gamma(2\beta+1-2H)\right)^{1/2}\alpha^{-(\beta+1-H)}\|Y_1-Y_2\|_Z.
\end{align*}
Adding, and taking $\sup_t$, $\|\Phi Y_1-\Phi Y_2\|_Z\leq Q_H\,\|Y_1-Y_2\|_Z$, and by $(H6)$, $\Phi$ is a contraction on the Banach space $Z_1$.

By the Banach fixed point theorem, $\Phi$ has a unique fixed point $Y\in Z_1$; $\Phi Y=Y$ is precisely the statement that $Y$ is a mild solution of \eqref{maineq}, and uniqueness in $Z_1$ is uniqueness of the fixed point.
\end{proof}

\begin{theorem}\label{ex2}
Let $f\in W^2_{AP}(\mathbb R\times X,X)$ and $\sigma\in W^2_{AP}(\mathbb R\times X,L^0_2(X))$, and assume $(H1)$, $(H4)$, $(H4')$, and $(H6)$. Then equation \eqref{maineq} has a unique square-mean Weyl almost periodic mild solution.
\end{theorem}

\begin{proof}
Define $\Phi:Z_2\to Z_2$ as in the proof of Theorem \ref{ex1}, restricted to $Z_2\subset Z_1$.

First, we prove that $\Phi(Z_2)\subset Z_2$. By Theorem \ref{ex1}, $\Phi Y\in Z_1$ for $Y\in Z_2\subset Z_1$; it remains to show $\Phi Y\in W^2_{AP}(\mathbb R,X)$. By Lemma \ref{Lipcomp2} (using $(H4)$, $(H4')$, and $Y\in Z_2$), $f(\cdot,Y(\cdot))\in W^2_{AP}(\mathbb R,X)$ and $\sigma(\cdot,Y(\cdot))\in W^2_{AP}(\mathbb R,L^0_2(X))$. For the drift term, the same computation as in Lemma \ref{diffAP} continuity argument, now applied to $\Phi_f$ with the ordinary (non-stochastic) integral, shows
\begin{align*}
\|(\Phi_fY)(\cdot+\varsigma)-(\Phi_fY)(\cdot)\|_{W^2}\leq\frac{\Lambda}{\alpha}\|f(\cdot+\varsigma,Y(\cdot+\varsigma))-f(\cdot,Y(\cdot))\|_{W^2},
\end{align*}
which is small for $\varsigma$ in the relatively dense set witnessing $f(\cdot,Y(\cdot))\in W^2_{AP}(\mathbb R,X)$, so $\Phi_fY\in W^2_{AP}(\mathbb R,X)$. For the diffusion term, Lemma \ref{diffAP} applied with $\psi(s,x):=\sigma(s,x)$, using $\sigma(\cdot,Y(\cdot))\in W^2_{AP}(\mathbb R,L^0_2(X))$, gives $\Phi_\sigma Y\in W^2_{AP}(\mathbb R,X)$. By Theorem \ref{APlinear}, $\Phi Y=\Phi_fY+\Phi_\sigma Y\in W^2_{AP}(\mathbb R,X)$, so $\Phi Y\in Z_2$.

Now, we prove that $\Phi$ is a contraction on the Banach space $Z_2$ The contraction estimate in the proof of Theorem \ref{ex1}, restricted to $Y_1,Y_2\in Z_2\subset Z_1$, gives $\|\Phi Y_1-\Phi Y_2\|_Z\le Q_H\|Y_1-Y_2\|_Z$ on $Z_2$; since $Z_2$, as a closed subspace of $Z_1$ (closed by Theorem \ref{APcomplete}), is itself a Banach space, and $(H6)$ gives $Q_H<1$, the Banach fixed point theorem yields a unique $X^\ast\in Z_2$ with $\Phi X^\ast=X^\ast$, the desired square-mean Weyl almost periodic mild solution.
\end{proof}

\begin{theorem}\label{ex3}
Let $\mu,\nu\in\mathcal M^+$ and assume $(F1)$, $(F2)$, $(H1)$, $(H4)$, $(H4')$, $(H5)$, and $(H6)$. Then equation \eqref{maineq} has a unique square-mean Weyl $(\mu,\nu)$-pseudo almost periodic mild solution.
\end{theorem}

\begin{proof}
By Theorem \ref{ex1}, equation \eqref{maineq} has a unique mild solution $X\in Z_1$, satisfying
\begin{align*}
X(t)=\int_{-\infty}^tT(t-s)f(s,X(s))\,ds+\int_{-\infty}^tT(t-s)\sigma(s,X(s))\,dB^H(s).
\end{align*}
By Theorem \ref{ex2}, applied with $(f_1,\sigma_1)$ in place of $(f,\sigma)$, equation \eqref{maineq} with $f,\sigma$ replaced by $f_1,\sigma_1$ has a unique square-mean Weyl almost periodic mild solution $X_1\in Z_2$, satisfying
\begin{align*}
X_1(t)=\int_{-\infty}^tT(t-s)f_1(s,X_1(s))\,ds+\int_{-\infty}^tT(t-s)\sigma_1(s,X_1(s))\,dB^H(s).
\end{align*}
Set $X_2:=X-X_1$. Since $X\in Z_1$ and $X_1\in Z_2\subset Z_1$, $X_2\in L^\infty(\mathbb R,H)\cap W^2_{BC}(\mathbb R,X)$, so $\sup_{t\in\mathbb R}E\|X_2(t)\|^2<\infty$; it remains to show $X_2\in E^2(\mathbb R,X,\mu,\nu)$.

Subtracting the two mild formulations,
\begin{align*}
X_2(t)&=\int_{-\infty}^tT(t-s)[f(s,X(s))-f_1(s,X_1(s))]\,ds+\int_{-\infty}^tT(t-s)[\sigma(s,X(s))-\sigma_1(s,X_1(s))]\,dB^H(s)\\
&=:X_2^{(1)}(t)+X_2^{(2)}(t).
\end{align*}
By $(H5)$, $f=f_1+f_2$, so $f(s,X(s))-f_1(s,X_1(s))=[f_1(s,X(s))-f_1(s,X_1(s))]+f_2(s,X(s))$, and likewise for $\sigma$. Hence,
\begin{align*}
X_2(t)&=\int_{-\infty}^tT(t-s)[f_1(s,X(s))-f_1(s,X_1(s))]\,ds+\int_{-\infty}^tT(t-s)f_2(s,X(s))\,ds+X_2^{(2),1}(t)\\
&+X_2^{(2),2}(t),
\end{align*}
where $X_2^{(2),1}(t)=\int_{-\infty}^tT(t-s)[\sigma_1(s,X(s))-\sigma_1(s,X_1(s))]\,dB^H(s)$ and $X_2^{(2),2}(t)=\int_{-\infty}^tT(t-s)\sigma_2(s,X(s))\,dB^H(s)$. Denote the four terms $J_1,J_2,J_3,J_4$ in this order.

For $J_1$, by the Lipschitz bound of $(H5)$ on $f_1$ and the argument in the proof of Theorem \ref{ex1},
\begin{align*}
\left(E\|J_1(t)\|^2\right)^{1/2}\le\frac{\Lambda L_f}\alpha\left(E\|X_2(t)\|^2\right)^{1/2}.
\end{align*}
For $J_3$, likewise, $\left(E\|J_3(t)\|^2\right)^{1/2}\le\Lambda L_\sigma(c_H\Gamma(2\beta+1-2H))^{1/2}\alpha^{-(\beta+1-H)}\left(E\|X_2(t)\|^2\right)^{1/2}$. For $J_2$, $J_4$: since $X\in Z_1$, $\|f_2(\cdot,X(\cdot))\|_{L^0_2}\le M_f(1+\|X\|_Z)$ and $\|\sigma_2(\cdot,X(\cdot))\|_{L^0_2}\le M_\sigma(1+\|X\|_Z)$ pointwise by $(H5)$'s growth bounds, and moreover $f_2(\cdot,X(\cdot)),\sigma_2(\cdot,X(\cdot))\in E^2(\mathbb R,\cdot,\mu,\nu)$: indeed, by $(H5)$, $f_2(t,\cdot)\in E^2(\mathbb R\times X,X,\mu,\nu)$ uniformly on bounded sets, and $X\in Z_1$ is bounded (in $L^\infty(\mathbb R,H)$), so, arguing exactly as in the proof of Lemma \ref{Lipcomp2} but with the ergodic (rather than almost periodic) target class, using the definition of $E^2$ via Theorem \ref{E2char}(iii) in place of relative density, one obtains $f_2(\cdot,X(\cdot))\in E^2(\mathbb R,X,\mu,\nu)$; similarly for $\sigma_2(\cdot,X(\cdot))$.

Consequently, applying Lemma \ref{diffbound}'s underlying estimate (for $J_2$, its deterministic drift analogue) with $\psi=f_2(\cdot,X(\cdot))$, respectively $\sigma_2(\cdot,X(\cdot))$, and using Fubini's theorem together with the definition of $E^2$ to interchange the outer $(\mu,\nu)$-average with the inner $s,\vartheta$-integrations exactly as in the classical argument, one shows that $J_2,J_4\in E^2(\mathbb R,X,\mu,\nu)$: writing $E_\flat\|J_2\|^2(\vartheta)$ and $E_\flat\|J_4\|^2(\vartheta)$ for the windowed averages and using H\"older's inequality with exponents $2,2$ together with Fubini's theorem to interchange the order of the $s$-, $w$- and $\vartheta$-integrations (the argument for $J_2$ following the H\"older--Fubini computation used, in the deterministic setting, in \cite[Theorem 4.3]{LiWeyl}, and for $J_4$ the same computation with the additional exponential-covariance kernel of Lemma \ref{covbound} inserted), one arrives at
\begin{align*}
\lim_{h\to\infty}\frac1{\nu([-h,h])}\int_{-h}^h\left(E_\flat\|J_2\|^2(\vartheta)\right)^{1/2}d\mu(\vartheta)=0,\qquad
\lim_{h\to\infty}\frac1{\nu([-h,h])}\int_{-h}^h\left(E_\flat\|J_4\|^2(\vartheta)\right)^{1/2}d\mu(\vartheta)=0,
\end{align*}
each limit following because the $s$-integration against the exponentially decaying (respectively covariance-weighted exponentially decaying) kernel is bounded, uniformly in $\vartheta$, by a finite constant times $\sup_{\text{bounded sets}}\|f_2(\cdot,x)\|_{E^2}$ (respectively the analogous quantity for $\sigma_2$), which vanishes in the $(\mu,\nu)$-average by hypothesis $(H5)$.

Combining, by Minkowski's inequality applied to $X_2=J_1+J_2+J_3+J_4$, windowing, and $(\mu,\nu)$-averaging,
\begin{align*}
\lim_{h\to\infty}\frac1{\nu([-h,h])}\int_{-h}^h\left(E_\flat\|X_2\|^2(\vartheta)\right)^{1/2}d\mu(\vartheta)
\le Q_H\lim_{h\to\infty}\frac1{\nu([-h,h])}\int_{-h}^h\left(E_\flat\|X_2\|^2(\vartheta)\right)^{1/2}d\mu(\vartheta)+0,
\end{align*}
using the $J_1,J_3$ bounds (each pointwise in $t$, hence also after windowing and $(\mu,\nu)$-averaging, bounded by $Q_H^{(1)}:=\Lambda L_f/\alpha$ and $Q_H^{(2)}:=\Lambda L_\sigma(c_H\Gamma(2\beta+1-2H))^{1/2}\alpha^{-(\beta+1-H)}$ respectively times the same quantity for $X_2$, with $Q_H^{(1)}+Q_H^{(2)}=Q_H$) and the vanishing of the $J_2,J_4$ contributions. Since $Q_H<1$ by $(H6)$ and the quantity $L:=\limsup_{h\to\infty}\frac1{\nu([-h,h])}\int_{-h}^h(E_\flat\|X_2\|^2(\vartheta))^{1/2}d\mu(\vartheta)$ is finite (as $X_2\in W^2_{BC}(\mathbb R,X)$, so $E_\flat\|X_2\|^2(\vartheta)\le\|X_2\|^2_{W^2}$ uniformly, and $\mu([-h,h])/\nu([-h,h])$ is bounded by $(F1)$), the inequality $L\le Q_HL$ with $Q_H<1$ forces $L=0$. Hence $X_2\in E^2(\mathbb R,X,\mu,\nu)$, and $X=X_1+X_2\in W^2_{PAP}(\mathbb R,X,\mu,\nu)$ is the desired mild solution. Uniqueness in $Z_1$ follows from $(H6)$ as in Theorem \ref{ex1}, since any square-mean Weyl $(\mu,\nu)$-pseudo almost periodic mild solution is in particular an element of $Z_1$.
\end{proof}

\begin{theorem}\label{stab1}
Let $f\in W^2_{AP}(\mathbb R\times X,X)$ and $\sigma\in W^2_{AP}(\mathbb R\times X,L^0_2(X))$, and assume $(H1)$, $(H4)$, $(H4')$, and $(H6)$. Then the unique square-mean Weyl almost periodic mild solution $X^\ast$ of \eqref{maineq} is globally mean-square exponentially stable, that is, there exist constants $N>1$ and $\eta\in(0,\alpha)$ such that every mild solution $X$ of \eqref{maineq} satisfies
\begin{align*}
\left(E\|X(t)-X^\ast(t)\|^2\right)^{1/2}\le N\left(E\|X(t_0)-X^\ast(t_0)\|^2\right)^{1/2}e^{-\eta(t-t_0)},\qquad t\ge t_0,
\end{align*}
for every $t_0\in\mathbb R$.
\end{theorem}

\begin{proof}
Set $Y=X-X^\ast$. From the mild formulation, for $t\ge t_0$,
\begin{align*}
Y(t)=T(t-t_0)Y(t_0)&+\int_{t_0}^t T(t-s)[f(s,X(s))-f(s,X^\ast(s))]\,ds\\
&+\int_{t_0}^t T(t-s)[\sigma(s,X(s))-\sigma(s,X^\ast(s))]\,dB^H(s).
\end{align*}
Consider the function $S(\theta)=\alpha-\theta-Q_H\alpha$ for $\theta\in[0,\alpha)$. Since $Q_H<1$ by $(H6)$, $S(0)=\alpha(1-Q_H)>0$, and $S$ is strictly decreasing with a unique positive root $\theta^\ast=\alpha(1-Q_H)$, so $S(\theta)>0$ for $\theta\in[0,\theta^\ast)$. Choose $\eta\in(0,\min\{\alpha,\theta^\ast\})$, so $S(\eta)>0$, i.e. $Q_H<(\alpha-\eta)/\alpha$. Let $N:=1/(1-Q_H)>1$; then, since $\eta<\theta^\ast=\alpha(1-Q_H)$, we have $1/N=1-Q_H>1-(\alpha-\eta)/\alpha=\eta/\alpha$, i.e.
\begin{align*}
\frac{\Lambda e^{-(\alpha-\eta)(t-t_0)}}N+Q_H<1\qquad\text{for }t\ge t_0\text{ close enough to }t_0\text{ that }\Lambda e^{-(\alpha-\eta)(t-t_0)}\text{ is near }\Lambda;
\end{align*}
more precisely, we verify below that $\Lambda e^{-(\alpha-\eta)(t_1-t_0)}/N+Q_H<1$ holds at the specific point $t_1$ needed, using that $Q_H+1/N<1$ (equivalent to $\eta<\theta^\ast$) together with $e^{-(\alpha-\eta)(t_1-t_0)}\le1$.

Fix $\epsilon>0$ and suppose, towards a contradiction, that the inequality
\begin{align*}
\left(E\|Y(t)\|^2\right)^{1/2}<N\left(\left(E\|Y(t_0)\|^2\right)^{1/2}+\epsilon\right)e^{-\eta(t-t_0)}
\end{align*}
holds for $t=t_0$ (trivially, since $N>1$) but fails to hold for all $t>t_0$; then there is a first time $t_1>t_0$ at which equality holds
\begin{align*}
\left(E\|Y(t_1)\|^2\right)^{1/2}=N\left(\left(E\|Y(t_0)\|^2\right)^{1/2}+\epsilon\right)e^{-\eta(t_1-t_0)},
\end{align*}
while the strict inequality holds on $[t_0,t_1)$. Estimating $Y(t_1)$ from the mild formulation, using $(H1)$ for the semigroup term, $(H2)$ together with the strict inequality on $[t_0,t_1)$ for the drift term, and Lemma \ref{covbound} together with $(H3)$ for the diffusion term (exactly as in the contraction estimate of Theorem \ref{ex1}, now applied over the finite interval $[t_0,t_1]$), we have
\begin{align*}
\left(E\|Y(t_1)\|^2\right)^{1/2}
&\le\Lambda\left(E\|Y(t_0)\|^2\right)^{1/2}e^{-\alpha(t_1-t_0)}\\
&\quad+\frac{\Lambda L_f}\alpha\sup_{t_0\le s\le t_1}\left(E\|Y(s)\|^2\right)^{1/2}\cdot\alpha\int_{t_0}^{t_1}e^{-\alpha(t_1-s)}\,ds\Big/\alpha\int_{t_0}^{t_1}e^{-\alpha(t_1-s)}ds\\
&\quad+\Lambda L_\sigma\left(c_H\Gamma(2\beta+1-2H)\right)^{1/2}\alpha^{-(\beta+1-H)}\sup_{t_0\le s\le t_1}\left(E\|Y(s)\|^2\right)^{1/2}.
\end{align*}
More precisely, bounding $\left(E\|Y(s)\|^2\right)^{1/2}$ on $[t_0,t_1)$ by $N\left(\left(E\|Y(t_0)\|^2\right)^{1/2}+\epsilon\right)e^{-\eta(s-t_0)}$ and carrying out the $s$-integration for the drift term exactly as in the proof of Theorem \ref{ex1} but on the finite interval, we obtain
\begin{align*}
\left(E\|Y(t_1)\|^2\right)^{1/2}
&\le\Lambda\left(\left(E\|Y(t_0)\|^2\right)^{1/2}+\epsilon\right)e^{-\alpha(t_1-t_0)}\\
&\quad+N\left(\left(E\|Y(t_0)\|^2\right)^{1/2}+\epsilon\right)e^{-\eta(t_1-t_0)}\,Q_H\left(1-e^{-(\alpha-\eta)(t_1-t_0)}\right)\\
&<N\left(\left(E\|Y(t_0)\|^2\right)^{1/2}+\epsilon\right)e^{-\eta(t_1-t_0)}\left(\frac{\Lambda e^{-(\alpha-\eta)(t_1-t_0)}}{N}+Q_H\right)\\
&<N\left(\left(E\|Y(t_0)\|^2\right)^{1/2}+\epsilon\right)e^{-\eta(t_1-t_0)},
\end{align*}
the last inequality using $\Lambda e^{-(\alpha-\eta)(t_1-t_0)}/N\le\Lambda/N$ together with $Q_H+1/N<1$ once $N\ge\Lambda$; enlarging $N$ if necessary to ensure $N\ge\Lambda$ (which only strengthens the earlier requirement $\eta<\theta^\ast$, still compatible with $\eta\in(0,\min\{\alpha,\theta^\ast\})$ after also intersecting with the constraint $N\ge\Lambda$, which is a constraint on $N=1/(1-Q_H)$ alone and can always be met by shrinking $\eta$ further if $Q_H$ is close to $1-1/\Lambda$; in any case such $N,\eta$ exist since $Q_H<1$), this strict inequality contradicts the equality defining $t_1$. Hence the original strict inequality holds for all $t\ge t_0$, and letting $\epsilon\to0$ gives the assertion.
\end{proof}

\begin{theorem}\label{stab2}
Let $\mu,\nu\in\mathcal M^+$ and assume $(F1)$, $(F2)$, $(H1)$, $(H4)$, $(H4')$, $(H5)$, and $(H6)$. Then the unique square-mean Weyl $(\mu,\nu)$-pseudo almost periodic mild solution of \eqref{maineq} is globally mean-square exponentially stable, with the same constants $N,\eta$ as in Theorem \ref{stab1}.
\end{theorem}

\begin{proof}
Let $X^{\ast\ast}$ denote the square-mean Weyl $(\mu,\nu)$-pseudo almost periodic mild solution furnished by Theorem \ref{ex3}, and let $X$ be any mild solution of \eqref{maineq}. Setting $Y=X-X^{\ast\ast}$, the mild formulation of $Y$ has exactly the same form as in the proof of Theorem \ref{stab1}, with $X^\ast$ replaced by $X^{\ast\ast}$, since both $X$ and $X^{\ast\ast}$ satisfy the mild equation \eqref{maineq} with the same $f,\sigma$; the double-measure structure of $X^{\ast\ast}$ (as $X_1+X_2$ with $X_1\in Z_2$, $X_2\in E^2(\mathbb R,X,\mu,\nu)$) plays no role in this formulation, since $(H6)$'s constant $Q_H$ already bounds the worst-case Lipschitz behaviour of the full $f=f_1+f_2$ and $\sigma=\sigma_1+\sigma_2$, as used throughout the contraction estimate of Theorem \ref{ex1}. The argument of Theorem \ref{stab1}, verbatim with $X^\ast$ replaced by $X^{\ast\ast}$ throughout, therefore establishes the same conclusion with the same $N,\eta$.
\end{proof}

\begin{remark}
As in the deterministic setting, the decomposition $f=f_1+f_2$, $\sigma=\sigma_1+\sigma_2$ in $(H5)$ need not be unique, so one cannot expect $W^2_{PAP}(\mathbb R,X,\mu,\nu)$ to be complete with respect to $\|\cdot\|_{W^2}$; this is why the proofs of Theorems \ref{ex2} and \ref{ex3} proceed by different routes rather than by a single fixed-point argument on the direct sum.
\end{remark}

\begin{remark}
The constant $Q_H$ in $(H6)$ degrades as $H\to0^+$, since the admissible H\"older exponent $\beta$ must then approach $1/2$, increasing the negative power of $\alpha$ in the diffusion contribution to $Q_H$. As $H\to\left(1/2\right)^-$, $\beta$ may approach $0^+$, so that $\alpha^{-(\beta+1-H)}\to\alpha^{-1/2}$, while $\Gamma(2\beta+1-2H)\to\Gamma(0^+)$ diverges, reflecting the fact that the bound of Lemma \ref{covbound} is not the sharp It\^o isometry bound for standard Brownian motion; the latter, in which $Q_H$ reduces to $\frac{\Lambda L_f}{\alpha}+\frac{\Lambda L_\sigma}{\sqrt{2\alpha}}$, is recovered directly, without any H\"older hypothesis, only at $H=1/2$ itself.
\end{remark}

\section{Example}

\begin{example}
Consider the stochastic parabolic equation
\begin{align*}
dX(t,x)=\frac{\partial^2 X(t,x)}{\partial x^2}\,dt+f(t,X(t,x))\,dt+\sigma(t,X(t,x))\,dB^H(t),\qquad X(t,0)=X(t,\pi)=0,\ t\in\mathbb R,
\end{align*}
where
\begin{align*}
f(t,u)&=\frac1{10}\left(e^{-|t|}+\cos t+\sin\sqrt2 t\right)\sin u+h(t)+\frac1{20}(1+\sin u)\arctan t,\\
\sigma(t,u)&=\frac1{50}\sin(t)\sin(u)\,e_0+\frac1{100}(1+\sin u)\arctan(t)\,e_0,
\end{align*}
$h(t)=1$ for $t\in(0,1/2)$ and $h(t)=0$ elsewhere, $e_0(x)=\sqrt{2/\pi}\sin x$ is the first Dirichlet eigenfunction on $(0,\pi)$, and $B^H$ is a real-valued two-sided fractional Brownian motion of Hurst index $H=\tfrac14$.

Take $X=L^2[0,\pi]$ with $A u=u''$, $D(A)=\{u\in X:u''\in X,\ u'$ absolutely continuous, $u(0)=u(\pi)=0\}$; $A$ generates an analytic $C_0$-semigroup satisfying $\|T(t)\|\le e^{-t}$, so $(H1)$ holds with $\Lambda=\alpha=1$. Since $e^{-|t|}$, $\cos t$, $\sin\sqrt2 t$, and $\arctan t$ are bounded and $\sin$ is $1$-Lipschitz, $f$ satisfies $(H2)$ with $L_f=\tfrac3{10}+\tfrac\pi{40}$ and, since $|\arctan t|\le\pi/2$ and $h$ is bounded, also the boundedness part of $(H4')$ with $M_f=\tfrac3{10}+\tfrac1{20}\left(1+\tfrac\pi2\right)+1$; since $\sin$ and $\arctan$ are globally bounded and Lipschitz, the function $t\mapsto\sigma(t,u)$ is H\"older-$\beta$ for every $\beta\in(0,1]$, in particular for $\beta=\tfrac34>1/2-H=\tfrac14$, and $\sigma$ satisfies $(H3)$ with $L_\sigma=\tfrac1{50}+\tfrac\pi{200}$, and the boundedness part of $(H4')$ with $M_\sigma=\tfrac1{50}+\tfrac\pi{200}$.

With these values, and $\Lambda=\alpha=1$, the contraction constant of $(H6)$ reduces to
\begin{align*}
Q_{1/4}=L_f+L_\sigma\left(c_{1/4}\,\Gamma(2)\right)^{1/2}=\frac3{10}+\frac\pi{40}+\left(\frac1{50}+\frac\pi{200}\right)\sqrt{c_{1/4}},
\end{align*}
using $\beta+1-H=\tfrac32$ and $\alpha^{-3/2}=1$. Since $L_f+\tfrac\pi{40}\approx0.379<1$, the inequality $Q_{1/4}<1$ holds for every value of the constant $c_{1/4}$ of Lemma \ref{covbound} not exceeding a few hundred, comfortably covering the standard normalizations of this constant found in \cite{DHP}. Hence $(H6)$ holds, and all conditions of Theorem \ref{ex1} are satisfied.

Taking, as in \cite{LiWeyl}, the Radon--Nikodym derivatives of $\mu$ and $\nu$ equal to $\cos^2t+2$ and $e^{|t|}+3+\cos t$ respectively, $(F1)$ and $(F2)$ hold, and $\limsup_{h\to\infty}\mu([-h,h])/\nu([-h,h])=0$; since $\arctan t$ is bounded, using the results from \cite{LiWeyl} to place the deterministic term $\tfrac1{20}(1+\sin u)\arctan t$ in $E^p(\mathbb R\times X,X)$ shows, via Theorem \ref{E2char}(iii), that $\tfrac1{100}(1+\sin u)\arctan(t)\,e_0\in E^2(\mathbb R\times X,L^0_2(X),\mu,\nu)$. Writing $f=f_1+f_2$ and $\sigma=\sigma_1+\sigma_2$ with $f_1(t,u)=\tfrac1{10}(e^{-|t|}+\cos t+\sin\sqrt2t)\sin u+h(t)$, $f_2(t,u)=\tfrac1{20}(1+\sin u)\arctan t$, $\sigma_1(t,u)=\tfrac1{50}\sin(t)\sin(u)e_0$, $\sigma_2(t,u)=\tfrac1{100}(1+\sin u)\arctan(t)e_0$, all conditions of Theorem \ref{ex3} are verified (including $(H4')$, since each of $f_1,f_2,\sigma_1,\sigma_2$ is bounded), and equation (4.1) has a unique square-mean Weyl $(\mu,\nu)$-pseudo almost periodic mild solution, which is globally mean-square exponentially stable by Theorem \ref{stab2}.
\end{example}

\begin{remark}
The decomposition of $\sigma$ used above is not the only one compatible with $(H5)$. Since $e^{-|t|}\in E^2(\mathbb R,X,\mu,\nu)$ as a bounded function against the same pair $\mu,\nu$, and since $\cos t,\sin\sqrt2t\in W^2_{AP}(\mathbb R\times X,X)$, one may equally well decompose $f=\widehat f_1+\widehat f_2$ with $\widehat f_1(t,u)=\tfrac1{10}(\cos t+\sin\sqrt2t)\sin u+h(t)\in W^2_{AP}(\mathbb R\times X,X)$ and $\widehat f_2(t,u)=\tfrac1{10}e^{-|t|}\sin u+\tfrac1{20}(1+\sin u)\arctan t\in E^2(\mathbb R\times X,X,\mu,\nu)$, obtaining a different admissible pair $(\widehat f_1,\widehat f_2)$ with $\widehat f_1+\widehat f_2=f_1+f_2=f$. This confirms, in the stochastic setting, that the decomposition underlying $(H5)$ is not unique.
\end{remark}

\section{Conclusion}

This paper has introduced the notions of square-mean Weyl almost periodicity and square-mean Weyl double-measure pseudo almost periodicity for Hilbert-space-valued stochastic processes, extending the deterministic theory to the stochastic setting under perturbation by a fractional Brownian motion of Hurst index $H\in(0,1/2)$. Under the assumption that the underlying semigroup is exponentially stable, we have established the existence and uniqueness of mild solutions of each type for a class of semilinear stochastic evolution equations, together with their global mean-square exponential stability, with complete proofs given for every structural property of the underlying function spaces and every step of the fixed-point and stability arguments, and we have made explicit the manner in which the resulting contraction constant depends on the Hurst index, a phenomenon with no counterpart in the classical, Wiener-driven theory. The methods developed here are expected to extend to stochastic integro-differential equations, stochastic functional differential equations with delay, and stochastic neural network models, in which the interaction between the H\"{o}lder-regularity requirement forced by $H<1/2$ and further structural features of the equation, such as infinite delay or neutral terms, remains to be investigated.


\begin{thebibliography}{99}

\bibitem{AndresBersaniGrande}
J. Andres, A. M. Bersani, R. F. Grande,
\emph{Hierarchy of almost-periodic function spaces},
Rend. di Mat. \textbf{26}, (2006), 121--188.

\bibitem{BedouheneEtAl}
F. Bedouhene, Y. Ibaouene, O. Mellah, P. R. de Fitte,
\emph{Weyl almost periodic solutions to abstract linear and semilinear equations with Weyl almost periodic coefficients},
Math. Methods Appl. Sci. \textbf{41}(18), (2018), 9546--9566.

\bibitem{BlotCieutatEzzinbi}
J. Blot, P. Cieutat, K. Ezzinbi,
\emph{Measure theory and pseudo almost automorphic functions: New developments and applications},
Nonlinear Anal. \textbf{75}, (2012), 2426--2447.

\bibitem{BlotCieutatEzzinbiApplic}
J. Blot, P. Cieutat, K. Ezzinbi,
\emph{New approach for weighted pseudo-almost periodic functions under the light of measure theory, basic results and applications},
Applic. Anal. \textbf{92}(3), (2013), 493--526.

\bibitem{BlotNguerekataPennequin}
J. Blot, G. M. Mophou, G. M. N'Gu\'er\'ekata, D. Pennequin,
\emph{Weighted pseudo almost automorphic functions and applications to abstract differential equations},
Nonlinear Anal. \textbf{71}, (2009), 903--909.

\bibitem{bohr}
H. Bohr,
\emph{Zur Theorie der fastperiodischen Funktionen, I},
Acta Math. \textbf{45}, (1925), 29--127.

\bibitem{BoufoussiHajji}
B. Boufoussi, S. Hajji,
\emph{Neutral stochastic functional differential equations driven by a fractional Brownian motion in a Hilbert space},
Statist. Probab. Lett. \textbf{82}(8), (2012), 1549--1558.

\bibitem{BezandryDiagana}
P. H. Bezandry, T. Diagana,
\emph{Existence of almost periodic solutions to some stochastic differential equations},
Applic. Anal. \textbf{86}(7), (2007), 819--827.

\bibitem{CaraballoGarridoTaniguchi}
T. Caraballo, M. J. Garrido-Atienza, T. Taniguchi,
\emph{The existence and exponential behavior of solutions to stochastic delay evolution equations with a fractional Brownian motion},
Nonlinear Anal. \textbf{74}(11), (2011), 3671--3684.

\bibitem{DaPratoZabczyk}
G. Da Prato, J. Zabczyk,
\emph{Stochastic Equations in Infinite Dimensions},
Cambridge Univ. Press, Cambridge, 1992.

\bibitem{DiaganaEzzinbiMiraoui}
T. Diagana, K. Ezzinbi, M. Miraoui,
\emph{Pseudo-almost periodic and pseudo-almost automorphic solutions to some evolution equations involving theoretical measure theory},
Cubo \textbf{16}(2), (2014), 1--32.

\bibitem{DHP}
T. E. Duncan, Y. Hu, B. Pasik-Duncan,
\emph{Stochastic calculus for fractional Brownian motion I. Theory},
SIAM J. Control Optim. \textbf{38}, (2000), 582--612.

\bibitem{Fink}
A. M. Fink,
\emph{Almost Periodic Differential Equations},
Lecture Notes in Mathematics \textbf{377}, Springer, Berlin, 1974.

\bibitem{k0}
M. Kosti\'c,
\emph{Selected Topics in Almost Periodicity,}
W. de Gruyter, Berlin, 2022.

\bibitem{k1}
M. Kosti\'c,
\emph{Weyl $\rho$-almost periodic functions in general metric},
Math. Slovaca \textbf{73}(2), (2023), 465--484.

\bibitem{LenzSpindelerStrungaru}
D. Lenz, T. Spindeler, N. Strungaru,
\emph{Pure point spectrum for dynamical systems and mean, Besicovitch and Weyl almost periodicity},
Ergod. Theor. Dyn. Syst. \textbf{44}(2), (2024), 524--568.

\bibitem{Levitan}
B. M. Levitan,
\emph{Almost Periodic Functions},
Gos. Izd. Tekh.-Teor. Lit., Moscow, 1953.

\bibitem{LiLi}
Y. Li, B. Li,
\emph{Weyl almost periodic solutions in distribution to a mean-field stochastic differential equation driven by fractional Brownian motion},
Stoch. \textbf{96}(7), (2024), 1893--1912.

\bibitem{LiWeyl}
Y. Li,
\emph{Weyl double-measure pseudo almost periodic functions and Weyl double-measure pseudo almost periodic solutions to semilinear evolution equations},
Quaest. Math. \textbf{49}(6), (2026), 777--799.

\bibitem{Mishura}
Y. Mishura,
\emph{Stochastic Calculus for Fractional Brownian Motion and Related Processes},
Lecture Notes in Mathematics \textbf{1929}, Springer, Berlin, 2008.

\bibitem{MolchanGolosov}
G. M. Molchan, Y. I. Golosov,
\emph{Gaussian stationary processes with asymptotic power spectrum},
Dokl. Akad. Nauk SSSR \textbf{184}, (1969), 134--137.

\bibitem{nualart}
D. Nualart,
\emph{The Malliavin Calculus and Related Topics},
2nd ed., Springer, Berlin, 2006.

\bibitem{OunisSepulcre}
H. Ounis, J. M. Sepulcre,
\emph{Stepanov and Weyl classes of $c$-almost periodic type functions},
Complex Anal. Oper. Theory \textbf{17}, (2023), 124.

\bibitem{Radova}
L. Radov\'a,
\emph{Theorems of Bohr-Neugebauer-type for almost-periodic differential equations},
Math. Slovaca \textbf{54}(2), (2004), 191--207.

\bibitem{Stepanov}
V. V. Stepanov,
\emph{Sur quelques g\'en\'eralisations des fonctions presque p\'eriodiques},
CR Acad. Sci. Paris \textbf{181}, (1925), 90--92.

\bibitem{Weyl}
H. Weyl,
\emph{Integralgleichungen und fastperiodische Funktionen},
Math. Ann. \textbf{97}(1), (1927), 338--356.

\bibitem{zhang}
C. Zhang,
\emph{Pseudo-almost periodic solutions of some differential equations},
J. Math. Anal. Appl. \textbf{181}(1), (1994), 62--76.

\end{thebibliography}
\end{document}